\documentclass[11pt]{article}
\usepackage[margin =1in]{geometry}
\usepackage{amssymb,amsthm,amsmath,mathrsfs}
\usepackage{enumerate}
\usepackage{hyperref}
\usepackage{cite}
\usepackage[capitalize]{cleveref}
\usepackage{enumitem}
\usepackage{color}
\usepackage{cancel}
\usepackage{tikz}
\usetikzlibrary{patterns}
\usepackage{pgf}
\usepackage{svg}
\usepackage{pgfplots}
\usepackage{import}

\usepackage[utf8]{inputenc} 
\usepackage[T1]{fontenc}    
\usepackage{hyperref}       
\usepackage{url}            
\usepackage{booktabs}       
\usepackage{amsfonts}       
\usepackage{nicefrac}       
\usepackage{microtype}
\usepackage{multirow}
\usepackage{xfrac}
\usepackage{todonotes}
\usepackage{titlesec}
\usepackage{caption}
\usepackage{bbm}
\usepackage{stmaryrd}

\usepackage{multirow}

\usepackage{caption}
\usepackage{float}

\pgfplotsset{compat=1.18}

\titleformat{\subsubsection}[runin]
{\normalfont\normalsize\bfseries}{\thesubsubsection}{1em}{}

\usepackage{graphicx}

\usepackage{appendix}
\usepackage[ruled]{algorithm2e}
\numberwithin{equation}{section}

\newtheorem{thm}{Theorem}[section]
\newtheorem{theorem}[thm]{Theorem}

\newtheorem{lemma}[thm]{Lemma}

\crefname{corollary}{Corollary}{Corollaries}
\crefname{algorithm}{Algorithm}{Algorithms}
\crefname{figure}{Figure}{Figures}
\crefname{algocf}{Algorithm}{Algorithms}

\newcommand{\I}{\mathcal{I}}

\newcommand{\argmin}{\operatornamewithlimits{argmin}}

\newcommand{\R}{\mathbb{R}}
\newcommand{\gradB}{\nabla_{\!\!{\scriptscriptstyle B}}}

\newcommand{\Ucoef}{\rho}

\newcommand{\Wcoef}{\gamma}

\newcommand{\Mem}{\mathcal{H}}
\newcommand{\MemExpr}{\{(x_i, f_i, g_i, z_{i+1}, \tau_i, L_i, \Delta_i)\}_{i=n-\idxMem}^{n-1}}

\newcommand{\idxArgmin}{m}
\newcommand{\idxLastSuccess}{s}
\newcommand{\idxMem}{k}

\newcommand{\idxSetSuccess}{J}

\newcommand{\epsFunc}{\epsilon(\Ucoef, \Wcoef)}
\newcommand{\ones}{\mathbf{1}}
\newcommand{\idxEpoch}{\ell}
\newcommand{\epoch}[1]{{#1}^{(\idxEpoch)}}
\newcommand{\dotB}[2]{\langle {#1}, {#2} \rangle_B}
\newcommand{\dotp}[2]{\langle {#1}, {#2} \rangle}
\newcommand{\normB}[1]{\|{#1}\|_B}

\newcommand{\deltaExpr}{L_n \tau_\idxLastSuccess\left(\frac{1}{L_\idxLastSuccess^2} - \frac{1}{L_n^2}\right) \frac{1}{2}\normB{g_\idxLastSuccess}^2}
\newcommand{\indicator}[1]{1_{#1}}
\newcommand{\ip}[2]{\sum_{i={n-\idxMem}}^{n-1} {#1}_i {#2}_i}
\newcommand{\fullSum}[1]{\sum_{i={n-\idxMem}}^{n-1} {#1}_i}

\newcommand{\norm}[1]{\|{#1}\|}

\newcommand{\AlgSpacingA}{\vspace{2.7mm}}
\newcommand{\AlgSpacingB}{\vspace{0mm}}

\theoremstyle{definition}
\newtheorem{remark}{Remark}

\theoremstyle{plain}

\theoremstyle{definition}

\allowdisplaybreaks
\begin{document}
    
    \title{A Practical Adaptive Subgame Perfect Gradient Method}

    \author{Alan Luner\footnote{Johns Hopkins University, Department of Applied Mathematics and Statistics, \url{aluner1@jhu.edu}} \qquad Benjamin Grimmer\footnote{Johns Hopkins University, Department of Applied Mathematics and Statistics, \url{grimmer@jhu.edu}}}
    \date{}
    \maketitle
    \begin{abstract}
        We present a performant gradient method for smooth convex optimization, drawing inspiration from several recent advances in the field. Our algorithm, the Adaptive Subgame Perfect Gradient Method (ASPGM) is based on the notion of subgame perfection, attaining a dynamic strengthening of minimax optimality. At each iteration, ASPGM makes a momentum-type update, optimized dynamically based on a (limited) memory/bundle of past first-order information. ASPGM is linesearch-free, parameter-free, and adaptive due to its use of recently developed auto-conditioning, restarting, and preconditioning ideas. We show that ASPGM is competitive with state-of-the-art L-BFGS methods on a wide range of smooth convex problems. Unlike quasi-Newton methods, however, our core algorithm underlying ASPGM has strong, subgame perfect, non-asymptotic guarantees, providing certificates of solution quality, resulting in simple stopping criteria and restarting conditions.
    \end{abstract}

    \section{Introduction}
There is a common trade-off in optimization between strong convergence guarantees and practical speed. For first-order convex optimization, linesearch methods, in particular quasi-Newton methods, are well-established and popular algorithms. When measured by practical performance, L-BFGS \cite{LBFGS_Nocedal} and its variants (e.g., \cite{LBFGSB_Zhu,oLBFGS_Schraudloph}) are state-of-the-art for many smooth convex problems, often converging superlinearly. However, despite this empirical performance, these methods have rather limited non-asymptotic guarantees. At the other extreme, accelerated gradient methods with comprehensive convex optimization theory are often outperformed by these linesearch methods. For example, the Optimized Gradient Method (OGM)~\cite{OGM} is known to provide the minimax optimal convergence guarantee for smooth convex problems but can fail to converge faster than its worst-case rate even on very simple instances\footnote{See the example of minimizing $f(x)=\tfrac{1}{2}x^2$ in~\cite{SPGM} where OGM fails to accelerate beyond the minimax worst-case rate.}.

In light of this dichotomy, there has been a recent renewed focus in smooth optimization on building performant methods with explicit non-asymptotic convergence rates. In this work, we bring together two separate recent directions of this movement: (i) adaptive methods that learn problem parameters and conditioning while running and (ii) subgame perfect methods that optimally (in a strong game-theoretic sense) leverage the bundle of first-order information collected at runtime.

Regarding (i), good performance of first-order methods often relies on an estimate of the problem's level of smoothness and its overall conditioning. Estimating smoothness is often addressed by backtracking to ensure an estimated constant satisfies the needed conditions for progress at each iteration~\cite{FISTA_BeckTeboulle,UFGM_Nesterov}. Recently, however, many works~\cite{AdaNAG_Suh,AdGD_Malitsky,AdGD_Prox2_Malitsky,ACFGM_Li,NAGFree_Cavalcanti,AdaptiveProximalLocalLipschitz_Latafat,AdaptivePrimalDual_Vladarean,AdaptiveAltMinimization_Latafat,LinesearchFreeNonconvex_Yagishita,AdaptiveUniversal_Oikonomidis,AutoconditionedProjected_LanLiXu} have produced linesearch-free methods with strong supporting convergence theory, capable of estimating problem smoothness while avoiding the cost of backtracking. Problem conditioning can be addressed through careful selection of a preconditioner to improve the problem's overall condition number. BFGS-type methods could be viewed as providing such a live preconditioner through their maintenance of a second-order approximation. Recently, an alternative online preconditioner was developed by~\cite{OnlineScaling_Gao}, extending the ideas of \cite{MultidimensionalBacktracking_Kunstner}, in the framework of online learning to optimize performance.

Regarding (ii), the performance of first-order methods has long been known to improve if methods are allowed to maintain an additional memory. Bundle methods \cite{Bundle_Kiwiel,BundleProximity_Kiwiel,VariantsOfBundleMethods,BundleLevel_Lan} have a long history of success, maintaining a bundle of many past first-order evaluations for use in its iterative update. The works \cite{OGMM_Florea,GMM_Nesterov}, among others, introduce several ``gradient methods with memory'', using a memory bundle to improve the progress of their estimate sequence and construct strategic, tight smooth lower bounds. Recently, the subgame perfect approach of~\cite{SPGM} provided a theoretical framework describing the optimal dynamic usage of memory in a gradient method, strengthening the typical standard of minimax optimality. Their resulting Subgame Perfect Gradient Method (SPGM) solves a planning subproblem at each iteration to optimally leverage a bundle of past information---although to do so optimally requires exact knowledge of the given convex problem's smoothness constant.

This work provides a new performant method, which we call the Adaptive Subgame Perfect Gradient Method (ASPGM), combining the benefits of adaptive linesearch-free methods, preconditioning, and a subgame perfect usage of memory.

\subsection{Related Literature}

\paragraph{Optimized Gradient Method (OGM)}
First proposed numerically by Drori and Teboulle \cite{FirstPEP} and later formalized by Kim and Fessler \cite{OGM}, OGM possesses an optimal $O(1/N^2)$ convergence guarantee for smooth convex minimization. This is the same order of convergence as Nesterov's classic Fast Gradient Method~\cite{Nesterov} but improves asymptotically by a constant factor of two. Drori~\cite{Drori_OGMOptimal} showed that OGM's convergence guarantee is exactly minimax optimal among all first-order methods. That is, among all gradient-span methods $A\in\mathcal{A}$ producing a final iterate $x_N$, OGM attains the minimum value of
\begin{equation}\label{eq:minimax-optimality}
    \min_{A\in \mathcal{A}} \max_{(f,x_0)\in\mathcal{F}_L} \frac{f(x_N)-f(x_\star)}{\frac{L}{2}\|x_0-x_\star\|^2}
\end{equation}
where $\mathcal{F}_L$ denotes the family of all $L$-smooth convex problems with a minimizer $x_\star$ and initialization $x_0$. Note that OGM is not adaptive or parameter-free as it relies on knowledge of the smoothness constant $L$. OGM and its theory are formally introduced in Section~\ref{SubSec:OGM}.

\paragraph{Subgame Perfect Gradient Method (SPGM)}
SPGM~\cite{SPGM} improves OGM to attain a stronger notion of optimality than minimax optimality in the sense of~\eqref{eq:minimax-optimality} called subgame perfection. Consider a gradient method that sees iterates and first-order information of $\{(x_i,f_i,g_i)\}_{i=0}^N$ where $f_i=f(x_i)$ and $g_i=\nabla f(x_i)$. At any iteration $n\in\{1,\dots, N\}$ having seen a history of $\mathcal{H}=\{(x_i,f_i,g_i)\}^{n-1}_{i=0}$, denote the subclass of $L$-smooth convex problems agreeing with all observed first-order information by $\mathcal{F}^\mathcal{H}_L$ and the subclass of gradient-span methods required to reproduce $x_0,\dots,x_{n-1}$ by $\mathcal{A}^\mathcal{H}$. Then we say an algorithm is subgame perfect if for any iteration $n$ and observed history $\mathcal{H}$, the method attains the minimum value of
\begin{equation}\label{eq:subgame-perfect}
    \min_{A\in \mathcal{A}^\mathcal{H}} \max_{(f,x_0)\in\mathcal{F}_L^\mathcal{H}} \frac{f(x_N)-f(x_\star)}{\frac{L}{2}\|x_0-x_\star\|^2}.
\end{equation}
SPGM, formally introduced in Section~\ref{SubSec:SPGM}, attains this heightened standard. At each iteration, it solves a simple low-dimensional second-order cone program (SOCP) to compute an optimal next step. This update rule is optimal with respect to the stored first-order information in memory. Numerically, \cite{SPGM} observed that when applied to a non-adversarial problem instance, SPGM greatly improves upon the performance of non-dynamic methods like OGM. However, like OGM, SPGM is not adaptive as it remains dependent on a smoothness parameter as input.

\paragraph{BFGS/L-BFGS}
Quasi-Newton methods \cite{BFGS-B,BFGS-F,BFGS-G,BFGS-S,LBFGS_Nocedal,LBFGSB_Zhu,oLBFGS_Schraudloph} attempt to find a minimizer for a function by applying Newton's method with an approximation of the inverse Hessian. At each iteration $n$, one performs a low-rank update to the inverse Hessian approximation, $B_n$, and a search direction is computed as $v = -B_n \nabla f(x_n)$. Once a descent direction is found, a linesearch method is applied to determine a stepsize and update the current iterate. Common options for linesearch methods include Armijo backtracking \cite{Armijo,NocedalWright}, Strong-Wolfe conditions \cite{Wolfe,NocedalWright}, and a modified robust method introduced by Hager and Zhang in \cite{HagerZhang}. In L-BFGS \cite{LBFGS_Nocedal}, instead of storing the full matrix $B_n$, one stores only a limited memory of past data. This allows the matrix update and search direction calculation to be performed using only vector-vector products, avoiding the storage of dense matrices. While L-BFGS performs extremely well in practice and often has (super)linear convergence in a neighborhood of the solution, it only provides limited non-asymptotic guarantees on solution quality\footnote{Although recent progress~\cite{Jin_BFGS,Rodomanov_BFGS,Jin_BFGS_SelfConcordant} has been made establishing non-asymptotic guarantees for standard BFGS with additional assumptions on the objective $f$.}.

\paragraph{Preconditioning}
Preconditioning can be viewed as selecting an inner product via a preconditioner $B$ to reduce the problem's difficulty, with respect to some measure. In smooth convex problems, a good preconditioner would reduce $ \frac{L_B}{2}\|x_0-x_\star\|_B^2$ as this controls convergence (where $L_B$ is a global smoothness constant of $f$ with respect to a modified inner product $\langle\cdot,\cdot\rangle_B$). In smooth strongly convex problems, one would like a preconditioner $B$ minimizing the condition number $L_B/\mu_B$. For $f\in C^2$, selecting $B$ as the inverse of $\nabla^2 f(x_\star)$ drives this condition number to one in a neighborhood of $x_\star$. One can view quasi-Newton methods then as gradient descent with its preconditioner being adjusted at runtime to roughly approximate this inverse Hessian. The work \cite{MultidimensionalBacktracking_Kunstner} introduces a multidimensional backtracking technique to update a diagonal preconditioner at runtime. As a modern advancement of this idea, the Online Gradient Scaling Method of~\cite{OnlineScaling_Gao,OSGMBest} uses the online learning regret bound toolbox to adjust the preconditioner at runtime. Asymptotically, this ensures an optimal preconditioned performance measure for gradient descent.

\paragraph{Adaptive Methods}
Adaptive methods offer strategies to address the issue of unknown function parameters. Backtracking is perhaps the most common and simple approach for smooth convex optimization when the smoothness parameter $L$ is unknown. At each iteration, the algorithm evaluates the current estimate of $L$, and backtracks ($L \leftarrow 2L$) if certain smoothness inequalities are not satisfied. In \cite{FISTA_BeckTeboulle}, Beck and Teboulle extended the accelerated rates of Nesterov's Fast Gradient Method \cite{Nesterov} to a backtracking linesearch setting. This approach was subsequently applied to broader settings, including \cite{UFGM_Nesterov} proposing a more general universal backtracking method for the broader class of functions with H\"{o}lder-continuous gradient. In \cite{OBL}, Park and Ryu propose the Optimal Backtracking Linesearch (OBL) method, which they show is minimax optimal against a certain generalized class of smooth convex gradient oracles.

While effective, a downside of backtracking is that it can be ``wasteful'' with respect to oracle evaluations. When a backtrack occurs, the new oracle information computed at the candidate iterate is effectively discarded.
Linesearch-free methods manage to avoid this overhead through various techniques. In bundle methods \cite{Bundle_Kiwiel,BundleProximity_Kiwiel, VariantsOfBundleMethods,BundleLevel_Lan}, this is handled with a ``null step'' in which the function value and gradient information are still saved in the bundle to inform future iterations. Starting with the introduction of adaptive gradient descent by~\cite{AdGD_Malitsky}, many recent works~\cite{AdaNAG_Suh,AdGD_Prox2_Malitsky,ACFGM_Li,AdaptiveProximalLocalLipschitz_Latafat,LinesearchFreeNonconvex_Yagishita,AdaptiveUniversal_Oikonomidis,AutoconditionedProjected_LanLiXu} adaptively learn the smoothness constant without backtracking, while also providing a non-asymptotic convergence guarantee. The task of designing methods that estimate an unknown level of strong convexity is typically much harder than adapting to smoothness. Recently, \cite{NAGFree_Cavalcanti} introduced a method for adaptively learning strong convexity constants while guaranteeing (suboptimal) linear global convergence. Our ASPGM will employ a similar estimate as a heuristic to benefit from strong convexity.

\subsection{Our Contributions}

This work provides a new performant gradient method bringing together the benefits from adaptive linesearch-free methods, preconditioning, and subgame perfect usage of memory. Our three main contributions are:
\begin{itemize}
    \item {\bf A Backtracking-free Subgame Perfect Gradient Method (BSPGM).} We first design a core parameter-free algorithm, BSPGM, that is subgame perfect against an appropriate class of first-order oracles. Although friendly to backtracking, BSPGM is linesearch-free; instead, it uses an efficient null-step, auto-conditioning approach, while providing dynamic, non-asymptotic convergence guarantees.
    \item {\bf An Adaptive Subgame Perfect Gradient Method (ASPGM).} The strong non-asymptotic convergence guarantees and certificates resulting from BSPGM's subgame perfection provide a robust restarting condition for the algorithm, ensuring a contraction under strong convexity. Combining this restarting with preconditioning gives our full algorithm, ASPGM.
    \item {\bf A Numerical Survey.} We conduct a numerical survey comparing BSPGM and ASPGM with a range of existing first-order methods from the literature across a range of smooth convex optimization problems. Our methods consistently outperform existing adaptive methods with strong theoretical support and perform comparably to state-of-the-art L-BFGS methods that lack such theory. These findings hold whether performance is measured in terms of oracle complexity or wall clock time. 
\end{itemize}

\paragraph{Outline}
\cref{Sec:Prelim} introduces key concepts for building our algorithm and theory. \cref{Sec:Algorithm} gradually presents the full algorithm and its various elements, while demonstrating convergence theory and the subgame perfection of the core procedure underlying ASPGM. \cref{Sec:Numerics} presents numerical experiments, comparing with a wide range of algorithms. Finally, \cref{Sec:Theory} contains deferred proofs of our theoretical results.

    \section{Preliminaries} \label{Sec:Prelim}

We consider unconstrained minimization problems of the form
\begin{equation} \label{Eqn:MinimizeF}
    \min_{x \in \R^d} f(x)
\end{equation}
where $f$ is convex and differentiable with gradient $\nabla f$ locally Lipschitz and with a minimizer $x_\star$. We highlight the fact that a global Lipschitz constant $L$ on the gradient need not be known, or even exist, for our algorithm to be well-defined and progress. We denote $I$ as the identity matrix, $e_i$ as the $i$th standard basis vector, and $\ones$ as the all ones vector, where the dimension of each should be clear from context.

Throughout, we will use $\langle \cdot, \cdot \rangle$ to denote the standard Euclidean inner product and $\|\cdot \|$ its corresponding norm. However, we will often use the modified inner products and norms defined by
\begin{equation}
    \dotB{x}{y} = \langle x, B^{-1} y \rangle \qquad \qquad \qquad \normB{x} = \sqrt{\dotB{x}{x}}
\end{equation}
for a given symmetric positive definite matrix $B$. Accordingly, we will use $\gradB f(x)$ to denote the gradient of $f$ at $x$, under the inner product $\langle \cdot, \cdot \rangle_B$. One can verify that $\gradB f(x) = B \nabla f(x)$. 
Our nonstandard choice of notation will become relevant when we apply preconditioning to our problem in \cref{SubSec:Precon}. However, for much of our analysis, one can assume that $B=I$, and the standard Euclidean forms are recovered.

We recall two key inequalities for smooth convex functions. First, we have
\begin{equation}\label{Eqn:ConvexityIneq}
    f(y) \geq f(x) + \dotB{\gradB f(x)}{y-x} \qquad \forall x,y
\end{equation}
which we will refer to as the \textit{convexity inequality}. Second, if $f$ has locally $L$-Lipschitz gradient near $x$, often referred to as local $L$-smoothness, we have
\begin{equation}\label{Eqn:CocoercivityIneq}
    f(y) \geq f(x) + \dotB{\gradB f(x)}{y-x} + \frac{1}{2L} \normB{\gradB f(x) - \gradB f(y)}^2 \qquad \forall y \text{ s.t. } \normB{x-y} < R
\end{equation}
which we will refer to as the \textit{cocoercivity inequality}. We also define the function 
\begin{equation}
    \tilde{L}_B(x,y) = \frac{\frac{1}{2} \normB{\gradB f(x) - \gradB f(y)}^2}{f(y) - f(x) - \dotB{\gradB f(x)}{y-x} }
\end{equation}
which returns the smallest $L$ such that \eqref{Eqn:CocoercivityIneq} holds from $x$ to $y$. We use the convention that $0/0 = 0$ to resolve the edge case where both numerator and denominator are zero, as any positive $L$ then satisfies~\eqref{Eqn:CocoercivityIneq}.

We will often consider a discretized set of data for our function. Letting $\I = \{0, 1, \dots, N, \star\}$ for some $N$, we consider $\{(x_i, f_i, g_i) \}_{i \in \I}$, where $f_i = f(x_i)$ and $g_i = \gradB f(x_i)$. We define the following values, corresponding to our inequalities \eqref{Eqn:ConvexityIneq} and \eqref{Eqn:CocoercivityIneq} above: for $i,j\in\I$
\begin{align} \label{Eqn:Def_WQV}
    W_{i,j} &:= f_i - f_j - \dotB{g_j}{x_i - x_j}, \\
    Q_{i,j}(L) &:= f_i - f_j - \dotB{g_j}{x_i - x_j} - \frac{1}{2L} \normB{g_i - g_j}^2.
\end{align}
Thus, for a locally smooth convex function and an appropriate $L$, it holds that $W_{i,j}$ and $Q_{i,j}(L)$ are nonnegative for all $i,j \in \I$.

We will also consider the case of smooth, strongly convex functions. Recall that any function $f$ that is $\mu$-strongly convex with respect to $\langle\cdot,\cdot \rangle_B$ has 
\begin{equation} \label{Eqn:SCIneq}
    f(y) \geq f(x) + \dotB{\gradB f(x)}{y-x} + \frac{\mu}{2}\normB{x-y}^2 \qquad \forall x,y.
\end{equation}
Similar to the smooth case, we define the function $\tilde{\mu}_B(x,y) = \frac{f(y) - f(x) - \dotB{\gradB f(x)}{y-x}}{\frac{1}{2}\normB{x-y}^2}$
which returns the largest $\mu$ such that \eqref{Eqn:SCIneq} holds from $x$ to $y$.

\subsection{Design of Optimized Methods} \label{SubSec:Prelim_Design}
The derivation of many first-order methods can be written as an inductive argument in which one builds a nonnegative quantity at each iteration. A useful review of these arguments, also called potential-function proofs, can be found in \cite{PotentialFunctionProofs_Bansal,d_Aspremont_2021}. This inductive view of algorithm design and analysis will be central to our development herein. Below, we introduce three methods from the recent literature in this style: the Optimized Gradient Method (OGM) \cite{OGM}, the Subgame Perfect Gradient Method (SPGM) \cite{SPGM}, and the Optimized Backtracking Linesearch method (OBL) \cite{OBL}. Our method, ASPGM, will build upon these methods and, in particular, their inductions.

\subsubsection{OGM - An Inductive View}  \label{SubSec:OGM}
Assume for now that the smoothness constant $L$ with respect to an inner product $\langle \cdot,\cdot \rangle_B$ is known. The Optimized Gradient Method \cite{OGM} can be viewed as a method for maintaining an induction of the form
\begin{equation*}
H_n := \tau_n \left(f_\star - f_n + \frac{1}{2L}\normB{g_n}^2\right) + \frac{L}{2}\normB{x_0-x_\star}^2 - \frac{L}{2}\normB{z_{n+1}-x_\star}^2 \geq 0
\end{equation*}
for some $\tau_n >0$ and iterate sequences $x_n,z_{n+1}$. The base case of this induction can be established for any given $x_0$ by setting
$ z_{1} = x_0 - \frac{2}{L} g_0,\ \tau_0 = 2 $
since we then have
$ H_0 = 2Q_{\star, 0}(L) \geq 0 . $
Then OGM iterates by setting
\begin{equation}
    (\tau_n, x_n, z_{n+1}) = \texttt{OGM\_Update}(\tau_{n-1}, x_{n-1}, g_{n-1}, z_n)
\end{equation}
according to \cref{Alg:OGMUpdate}, where the modified final step formula applies either when the given iteration budget is exhausted ($n=N$) or when some stopping/restarting criterion is met.

\begin{algorithm}
    \caption{OGM Update Procedure} \label{Alg:OGMUpdate}
    \SetKwFunction{OGMUpdate}{OGM\_Update}
    \SetKwProg{Fn}{Function}{:}{\KwRet $(\tau_n,x_n,z_{n+1})$}
    \Fn{\OGMUpdate{$\hat{\tau}$, $x$, $g$, $z$}}
        {
        $\tau_n = \begin{cases}
            \hat{\tau} + \frac{1 + \sqrt{1+ 4\hat{\tau}}}{2} \qquad & \text{if \textbf{final step}} \\
            \hat{\tau} + 1 + \sqrt{1+2\hat{\tau}} \qquad & \text{otherwise}
        \end{cases}$  \\
        $x_n = \frac{\hat{\tau}}{\tau_n} (x - \frac{1}{L} g) + \frac{\tau_n - \hat{\tau}}{\tau_n}z$\\
        $z_{n+1} = z - \frac{\tau_{n}-\hat{\tau}}{L}g$ \\
        }
\end{algorithm}

\noindent These choices precisely ensure that the induction $H_n \geq 0$ continues, as one can verify that
\begin{equation}
    H_n = H_{n-1} + \tau_{n-1}Q_{n-1,n}(L) + (\tau_n-\tau_{n-1}) Q_{\star, n}(L) \geq 0.
\end{equation}
Equality can be shown above by expanding the polynomial $H_{n-1} + \tau_{n-1}Q_{n-1,n}(L) + (\tau_n-\tau_{n-1}) Q_{\star, n}(L)$ and collecting terms, and nonnegativity follows since each term in the sum is nonnegative.
A slight modification to this induction and algorithm is needed at the final iteration $N$, as one instead defines
\begin{equation} \label{Eqn:HN_Nonnegative}
H_N := \tau_N (f_\star - f_N) + \frac{L}{2}\normB{x_0-x_\star}^2 - \frac{L}{2}\normB{z_{N+1}-x_\star}^2 \geq 0.
\end{equation}
Rearranging \eqref{Eqn:HN_Nonnegative}, one arrives at a convergence rate of 
$ f_N - f_\star \leq \frac{\frac{L}{2}\normB{x_0-x_\star}^2}{\tau_N}$. This convergence guarantee is exactly minimax optimal among all gradient-span methods in the sense of~\eqref{eq:minimax-optimality}, proven by the matching lower bound of Drori~\cite{Drori_OGMOptimal}.

\subsubsection{SPGM - A Dynamically Optimized Induction}  \label{SubSec:SPGM}
Continuing to assume that the smoothness constant $L$ with respect to an inner product $\dotB{\cdot}{\cdot}$ is known, the Subgame Perfect Gradient Method \cite{SPGM}, and its limited memory variant, can be viewed as a dynamic improvement of OGM to use a stronger inductive hypothesis at each step. Rather than building $H_n\geq 0$ using the previous hypothesis $H_{n-1} \geq 0$, one could imagine an enhanced induction, constructing and using the strongest hypothesis of the form
\begin{equation} \label{Eqn:HPrimeDef}
    H' = \tau'\left(f_\star - f_m + \frac{1}{2L}\normB{g_m}^2\right) + \frac{L}{2}\normB{x_0-x_\star}^2 - \frac{L}{2}\normB{z' - x_\star}^2
\end{equation}
for some $\tau',z'$ and index $m$. Given a limited memory of size $k$, SPGM constructs such an aggregate nonnegative hypothesis by setting
\begin{equation} \label{Eqn:HPrimeInduction}
    H' = \sum_{i=n-\idxMem}^{n-1} \Ucoef_i H_i + \sum_{i=n-\idxMem}^{n-1} \Wcoef_i Q_{\star, i}(L) + \epsilon
\end{equation}
for nonnegative $\Ucoef,\Wcoef \in \R^k$ and $\epsilon \geq 0$. 

Let $Z \in \R^{d \times \idxMem}$ be the matrix with columns given by $z_{i+1}-x_0$ and indexed by $[n-\idxMem,n-1]$. Similarly, let $G \in \R^{d \times \idxMem}$ be the matrix with columns given by $\frac{g_i}{L}$ and indexed by $[n-\idxMem,n-1]$. Additionally, let $\tau = (\tau_{n-\idxMem}, \dots, \tau_{n-1})$ and let $v \in \R^\idxMem$ be defined componentwise by
\begin{equation*}
    v_i := f_i - \frac{1}{2L}\normB{g_i}^2.
\end{equation*}
Combining \eqref{Eqn:HPrimeDef} and \eqref{Eqn:HPrimeInduction}, one finds that any feasible hypothesis must set $z' = x_0 + Z\Ucoef - G\Wcoef$. Then, to maximize the value of $\tau'$ while keeping $H' \geq 0$, one should set
$ m \in\argmin_{i \in [n-\idxMem, n-1]} \{f_i - \frac{1}{2L}\normB{g_i}^2\}$. Using this index $\idxArgmin$, we further define vectors $q,r \in \R^\idxMem$ with components
\begin{align*}
    q_i & := \tau_i(f_i - \frac{1}{2L}\normB{g_i}^2) + \frac{L}{2}\normB{z_{i+1} - x_0}^2 - v_m \tau_i \\
    r_i & := f_i - \dotB{g_i}{x_i - x_0} + \frac{1}{2L}\normB{g_i}^2 - v_m.
\end{align*}
Under this selection, rearranging~\eqref{Eqn:HPrimeInduction} to solve for $\epsilon$, one obtains that 
\begin{equation*}
    \epsilon(\Ucoef,\Wcoef) = \ip{\Ucoef}{q} + \ip{\Wcoef}{r} - \frac{L}{2} \normB{Z \Ucoef - G \Wcoef}^2.
\end{equation*}
Then the strongest such hypothesis (i.e., the one maximizing $\tau'$) is attained by solving the second-order cone program of
\begin{equation} \label{Eqn:SOCP_OGM}
    \tau' = \begin{cases}
        \max & \ip{\Ucoef}{\tau} + \fullSum{\Wcoef} \\
        \mathrm{s.t.} & \epsilon(\Ucoef, \Wcoef) \geq 0\\
        & \Ucoef, \Wcoef \geq 0.
        \end{cases}
\end{equation}
Note that setting $\Ucoef=(0,\dots,0,1)$, $\Wcoef=0$, $\epsilon=0$, is feasible and recovers $H'=H_{n-1}$ when $m=n-1$, i.e., the hypothesis used by OGM. Thus $\tau' \geq \tau_{n-1}$, ensuring SPGM's intermediate aggregate hypothesis is always at least as strong as the minimax optimal induction of OGM.

SPGM then iterates by repeatedly constructing this optimized aggregate hypothesis and then applying the OGM step~(\cref{Alg:OGMUpdate}) from it to construct $H_n \geq 0$. This is formalized in \cref{Alg:SPGM}. Grimmer, Shu, and Wang~\cite{SPGM} established that SPGM's overall induction is at least as strong as OGM's and therefore SPGM is also exactly minimax optimal.
Further when a full memory is used, they constructed a dynamic adversarial strategy for revealing first-order information, proving a matching lower bound on the possible performance of any first-order method. Since this dynamic lower bound exactly matched SPGM's induction for every $n$, the authors conclude its dynamic guarantees are the strongest possible on the final objective gap given any history of first-order observations. That is, SPGM is subgame perfect in the sense of~\eqref{eq:subgame-perfect}.

\begin{algorithm}
\caption{SPGM} \label{Alg:SPGM} 
    \SetKwInOut{Input}{Input}
    \SetKwInOut{Output}{Output}
    \Input{ Convex function $f$, memory size $k \in \mathbb{N}$, $x_0 \in \R^d$, $L > 0$, $B \succ 0$}
    Set $\tau_0 = 2, \ z_1 = x_0 - \frac{2}{L} \gradB f(x_0)$\\
    \For{$n = 1,\dots,N$}
        {Compute $f_{n-1} = f(x_{n-1})$, $g_{n-1} = \gradB f(x_{n-1})$ \\
        Update memory $\Mem = \{(x_i,f_i,g_i,z_{i+1},\tau_i)\}_{i=n-k}^{n-1}$ and construct $Z$ and $G$ \\
        Set $m\in\argmin_{i \in [n-\idxMem, n-1]}\{f_i - \frac{1}{2L}\|g_i\|^2_B\}$ \\
        \If{\text{\eqref{Eqn:SOCP_OGM} is unbounded}}
            {Set $x_n = x_m - \frac{1}{L} g_m$ and \textbf{break}}
        Compute $\Ucoef,\Wcoef$ by solving~\eqref{Eqn:SOCP_OGM} and set $z' = x_0 + Z\Ucoef - G\Wcoef$ and $\tau' = \ip{ \Ucoef}{\tau} + \fullSum{\Wcoef} $ \\
        Set $(\tau_n,x_n,z_{n+1}) = \texttt{OGM\_Update}(\tau', x_m, g_m, z')$ according to \cref{Alg:OGMUpdate}}
    \Output{ $x_n$}
\end{algorithm}

\subsubsection{OBL - A Backtracking-Friendly Induction} \label{SubSec:OBL}
Both of the above methods rely on knowledge of the global smoothness constant $L$, which can be both unrealistic in practice and lead to overly conservative algorithms. One classic approach to avoiding this is the use of backtracking, where at each iteration an estimate $L_n$ is utilized to compute a step. Then, whichever inequalities are needed by the proof are checked; if they hold the step is accepted, otherwise $L_n$ is increased, typically to $2L_n$, and the process repeats.

To apply this reasoning to OGM, one would need to verify at each step the nonnegativity of $Q_{n-1,n}(L_n)$ and $Q_{\star, n}(L_n)$. This first quantity, $Q_{n-1,n}(L_n)$, is easy to compute at runtime. However, the second quantity $Q_{\star, n}(L_n)$ requires knowledge of $x_\star$ to be evaluated and hence prevents backtracking. The work of Park and Ryu~\cite{OBL} showed this issue can be circumvented by replacing the cocoercivity inequality $Q_{\star, n}(L_n)$ in OGM's induction with the weaker convexity inequality $W_{\star,n}$, which is independent of $L_n$. Carrying out this slightly weaker induction, one arrives at the backtracking-friendly OBL algorithm of~\cite{OBL}.

Each iteration maintains an inductive hypothesis of the form
\begin{equation} \label{Eqn:Def_U}
U_n := \tau_n \left(f_\star - f_n + \frac{1}{2L_n}\normB{g_n}^2\right) + \frac{L_n}{2}\normB{x_0-x_\star}^2 - \frac{L_n}{2}\normB{z_{n+1}-x_\star}^2 +\Delta_n \geq 0
\end{equation}
where the additional nonnegative constant $\Delta_n$ above accumulates error terms due to changing $L_n$. 
As a base case, one can initialize this induction with $\tau_0 = 1, z_1 = x_0 - \frac{1}{L_0} g_0$ to obtain $U_0 = W_{\star,0} \geq 0$.
Then, given $U_{n-1}\geq 0$ for some $\tau_{n-1},x_{n-1},z_n, \Delta_{n-1}$, OBL generates a test iterate by setting
\begin{equation} \label{Eqn:StandardOBLUpdate}
    (\tau_n, x_n, z_{n+1}, \Delta_n) = \texttt{OBL\_Update}(\tau_{n-1}, x_{n-1}, g_{n-1}, z_n, L_n, \frac{L_n}{L_{n-1}} \Delta_{n-1}, \delta_n)
\end{equation}
according to \cref{Alg:OBLUpdate}, where $\delta_n = L_n \tau_{n-1}\left(\frac{1}{L_{n-1}^2} - \frac{1}{L_n^2}\right) \frac{1}{2}\normB{g_{n-1}}^2$. If $Q_{n-1,n}(L_n) \geq 0$, then the iterate is accepted. Otherwise, one backtracks, setting $L_n \leftarrow 2 L_n$, discarding the computed $(x_n,f_n,g_n)$, and repeating \cref{Alg:OBLUpdate} until an iterate is accepted.

\begin{algorithm}
    \caption{OBL Update Procedure} \label{Alg:OBLUpdate}
    \SetKwFunction{OBLUpdate}{OBL\_Update}
    \SetKwProg{Fn}{Function}{:}{\KwRet $(\tau_n,x_n,z_{n+1},\Delta_n)$}
    \Fn{\OBLUpdate{$\hat{\tau}$, $x$, $g$, $z$, $L$, $\hat{\Delta}$, $\hat{\delta}$}}
        {
        $\tau_n = \begin{cases} 
            \hat{\tau} + \sqrt{\hat{\tau}} \qquad & \text{if \textbf{final step}} \\
            \hat{\tau} + \frac{1 + \sqrt{1+8\hat{\tau}}}{2} \qquad & \text{otherwise}
            \end{cases}$ \\
        $x_n = \frac{\hat{\tau}}{\tau_n} (x - \frac{1}{L} g) + \frac{\tau_n - \hat{\tau}}{\tau_n}z$\\
        $z_{n+1} = z - \frac{\tau_{n}-\hat{\tau}}{L}g$ \\
        $\Delta_n = \hat{\Delta} + \hat{\delta}$
        }
\end{algorithm}
Following this, one has that (see \cite[Theorem 6]{OBL})
\begin{equation*}
    U_n = \frac{L_n}{L_{n-1}} U_{n-1} + \tau_{n-1} Q_{n-1,n}(L_n) + (\tau_n - \tau_{n-1}) W_{\star, n} + \tau_{n-1}\left(\frac{L_n}{L_{n-1}} - 1 \right) (f_{n-1} - f_\star) \geq 0
\end{equation*}
where the equality is verified by expanding polynomial terms, and the inequality follows as each term is nonnegative. Hence, after $N$ iterations (and again modifying our final induction $U_N$) this backtracking provides a guarantee of
$ f_N - f_\star \leq \frac{\frac{L_N}{2}\normB{x_0-x_\star}^2 + \Delta_N}{\tau_N}$.
If a valid smoothness constant $L_0 = L$ is given initially, one will have $\Delta_N=0$. In this setting, \cite{OBL} proved this is the best possible convergence rate any first-order method can attain against a first-order oracle whose responses satisfy the inequalities $\{Q_{i-1,i}\}_{i = 1}^n \cup \{W_{\star, i}\}_{i=0}^n$.

    \section{Algorithm} \label{Sec:Algorithm}

In this section, we gradually build our algorithm ASPGM. In \cref{SubSec:Dynamic}, we build a dynamic optimization approach improving OBL's induction, mirroring that of SPGM, but with $L$ unknown. In \cref{SubSec:LSFree}, we demonstrate how this can be extended to a linesearch-free method we call the Backtracking-free Subgame Perfect Gradient Method (BSPGM). Section~\ref{SubSec:BSPGM_theory} presents convergence guarantees for BSPGM, including its subgame perfection. Finally, in \cref{SubSec:AdRestart} and \cref{SubSec:Precon}, we develop adaptive restarting and preconditioning and present our full algorithm ASPGM, using BSPGM as a subroutine.

\subsection{Dynamically Optimized Induction} \label{SubSec:Dynamic}

We consider the case of minimizing a locally smooth function, with unknown smoothness constant $L$. Additionally, we suppose that we have access to a memory of size $\idxMem$ of past algorithm information: $\Mem = \MemExpr$.

We begin by introducing some notation. Given $L_n > 0$, define the vector $v \in \R^\idxMem$ by its components:
\begin{align} \label{Eqn:Def_hwv}
    v_i & := f_i - \frac{1}{2L_n}\normB{g_i}^2
\end{align}
for $i \in [n-\idxMem, n-1]$.
Next, given a vector $\tau = (\tau_{n-\idxMem}, \dots, \tau_{n-1})$ and letting $\idxSetSuccess = \{i \in [n-\idxMem, n-1] \mid \tau_i > 0\}$, define two special indices $\idxArgmin$ and $\idxLastSuccess$ according to
\begin{align} 
    & \idxArgmin \in \argmin_{i \in \idxSetSuccess} v_i \label{Eqn:Def_idx_m}, \\
    & \idxLastSuccess = \max\{i \in \idxSetSuccess\} \label{Eqn:Def_idx_s} .
\end{align}
Colloquially, $\idxArgmin$ represents the ``best'' iterate so far, with respect to our induction $U_i$. As we will see, $x_\idxArgmin$ will be important in determining our next iterate $x_n$. On the other hand, $\idxLastSuccess$ corresponds to the most recent iterate for which our induction $U_i$ was nontrivial ($\tau_i\neq 0$). This index will play a key role in our induction and error term collection. Note that for each $i\not\in \idxSetSuccess$, we will have $U_i \equiv 0$. Further define $\delta_n$ as
\begin{equation}\label{Eqn:deltaExpr}
    \delta_n := \deltaExpr
\end{equation}
and the vectors $a,b \in \R^\idxMem$ with components
\begin{align}
    a_i & := \tau_i(f_i - \frac{1}{2L_i}\normB{g_i}^2) + \frac{L_i}{2}\normB{z_{i+1} - x_0}^2 - v_m \tau_i \\
    b_i & := f_i - \dotB{g_i}{x_i - x_0} - v_m
\end{align}
for $i\in[n-k,n-1]$.
Note that $v_i$ is dependent on the current $L_n$, while $a_i$ and $b_i$ also depend on the past value $L_i$.

We recall the inductive quantity $U_i$ \eqref{Eqn:Def_U} from OBL and now apply the dynamic induction strategy introduced in \cref{SubSec:OBL}. We define
\begin{equation} \label{Eqn:Def_UPrime}
    U' = \tau'\left(f_\star - f_m + \frac{1}{2 L_n}\normB{g_m}^2 \right) + \frac{L_n}{2}\normB{x_0 - x_\star}^2 - \frac{L_n}{2}\normB{z' - x_\star}^2 + \Delta' + \delta_n
\end{equation}
for some $\tau',z'$, and $\Delta' \geq 0$. We construct our enhanced hypothesis $U'$ by setting
\begin{equation} \label{Eqn:UPrimeInduction}
    U' = \sum_{i=n-\idxMem}^{n-1} \Ucoef_i U_i + \sum_{i=n-\idxMem}^{n-1}\Wcoef_i W_{\star,i} + \epsilon
\end{equation}
for nonnegative $\Ucoef, \Wcoef \in \R^k$ and $ \epsilon \geq 0$.
Proceeding as in SPGM, we seek to maximize $\tau'$ while keeping the error term $\Delta'$ bounded. Here we briefly summarize how this dynamic optimization is achieved, but we include a more detailed derivation in \cref{SubSec:ConvRateProof}.

Let $Z \in \R^{d \times \idxMem}$ be the matrix with columns given by $\frac{L_i}{L_n}(z_{i+1}-x_0)$ and indexed by $[n-\idxMem,n-1]$. Similarly, let $G \in \R^{d \times \idxMem}$ be the matrix with columns now given by $\frac{g_i}{L_n}$ and indexed by $[n-\idxMem,n-1]$.
Once again, by rearranging \eqref{Eqn:UPrimeInduction} and applying a bound on $\Delta'$, we find that the following constraints  must be satisfied:
\begin{align}
    z' & = x_0 + Z \Ucoef - G \Wcoef,  \label{Eqn:zPrime} \\ 
    \tau' &= \ip{\Ucoef}{\tau} + \fullSum{\Wcoef}, \\
    \epsilon &= \epsFunc =  \ip{\Ucoef}{a} + \ip{\Wcoef}{b} + \delta_n - \frac{L_n}{2} \normB{Z \Ucoef - G \Wcoef}^2. \label{Eqn:epsFunc}
\end{align}
In order to guarantee the nonnegativity of $U'$ we must have $\epsFunc \geq 0$. Thus, our dynamic optimization becomes the following subproblem with a single second-order cone constraint:
\begin{equation} \label{Eqn:SOCP}
    \tau' = \left\{\begin{array}{ll}
            \max_{\Ucoef, \Wcoef} & \ip{\Ucoef}{\tau} + \fullSum{\Wcoef} \\
            \text{s.t.} & \epsFunc \geq 0 \\
            & \Ucoef, \, \Wcoef \geq 0. \\
            \end{array} \right.
\end{equation}

Observe that maximizing $\tau'$ in \eqref{Eqn:SOCP} once again reduces to solving a low-dimensional second order cone program (SOCP). We emphasize that this SOCP is independent of the problem dimension $d$. Specifically, the subproblem \eqref{Eqn:SOCP} has dimension $2\idxMem$, where $\idxMem$ is the chosen memory size. By appropriately tracking and storing $O(\idxMem^2)$ past vector-vector products (i.e., $Z^T Z, G^T G, G^T Z$), we eliminate any matrix-vector products with full-dimension matrices. Additionally, the problem has a linear objective with a single quadratic constraint, further reducing the complexity of the subproblem. Thus, for high-dimensional ($k \ll d$) problems with moderate costs for gradient and/or function evaluations, the per iteration computational cost of \eqref{Eqn:SOCP} becomes negligible. 

\subsection{A Backtracking-free Subgame Perfect Gradient Method (BSPGM)} \label{SubSec:LSFree}
This dynamic induction already suggests a reasonable backtracking method. Given $\tau'$ and $z'$ after solving \eqref{Eqn:SOCP}, generate $x_n$ (along with $\tau_n, z_{n+1}, \Delta_n$) according to \eqref{Eqn:StandardOBLUpdate}. We then check if the current estimate of $L_n$ is valid, i.e., if $Q_{\idxArgmin, n}(L_n) \geq 0$. If so, accept $x_n$ and continue. If not, reject $x_n$, increase $L_n \leftarrow 2L_n$, and repeat the update rule. Note that when $x_n$ is accepted, this indicates that our hypothesis $U_n \geq 0$ is valid, so the induction can continue.

While this approach represents an effective algorithm, it is inefficient in terms of how it uses new oracle information. When a proposed next iterate $x_n$ is rejected, the new gradient $g_n$ and function value $f_n$ are effectively discarded without providing additional information beyond the fact that the current estimate $L_n$ is too small. For high-dimensional problems, first-order oracles can have a large computational cost; we therefore want to learn as much information as possible from each oracle call.

In the case that $Q_{\idxArgmin,n}(L_n) < 0$, the computed first-order information still satisfies convexity, ensuring $W_{\star,n} = f_\star - f_n - \langle g_n, x_\star - x_n \rangle \geq 0$. As an alternate strategy to backtracking in this case, we can accept the computed iterate $x_n$, adding its function value and gradient to our history, but modify our induction to avoid usage of $Q_{\idxArgmin,n}(L_n)$. In particular, we set $\tau_n = 0, z_{n+1} = x_0, \Delta_n = 0$ to ensure that $U_n \geq 0$ holds vacuously, as
\begin{align*}
    U_n & = \tau_n\left(f_\star - f_n + \frac{1}{2L_n}\normB{g_n}^2\right) + \frac{L_n}{2}\normB{x_0-x_\star}^2-\frac{L_n}{2}\normB{z_{n+1}-x_\star}^2 + \Delta_n \\
    & = 0 + \frac{L_n}{2}\normB{x_0-x_\star}^2-\frac{L_n}{2}\normB{x_0-x_\star}^2 + 0 \\
    & = 0.
\end{align*}
We therefore continue our induction of $U_n \geq 0$ with no dependence on the negative term $Q_{m,n}(L_n)$.

This strategy is similar to the typical bundle method approach of allowing both ``serious steps'' where the main iterate sequence updates and ``null steps'' where only the bundle of first-order information is improved.
\cref{Alg:BaseAlg} formalizes our use of this conditional approach. We emphasize that while there are remnants of a backtracking mechanism in how we iteratively increase $L_n$, this method is notionally distinct as we are not discarding any computed $x_n$ and associated first-order oracle values when $L_n$ increases. Rather, $(x_n, f_n, g_n)$ are incorporated into our bundle of historical information via the nonnegative quantity $W_{\star,n}$. We also note the use of $\tilde{L}_B(x_m, x_n)$ to increase $L_n$ matches the auto-conditioning approach used by \cite{AdGD_Malitsky,ACFGM_Li,AdaNAG_Suh} and others.

\begin{algorithm}[ht]
\caption{BSPGM: A Backtracking-free Subgame Perfect Gradient Method} \label{Alg:BaseAlg} 
    \SetKwInOut{Input}{Input}
    \SetKwInOut{Output}{Output}
    \Input{ Convex, locally smooth function $f$, subgame memory size $k \in \mathbb{N}$, $x_0 \in \R^d$, $L_0 > 0$, $B \succ 0$}
    Set $\tau_0 = 1, z_1 = x_0 - \frac{1}{L_0} \gradB f(x_0), L_1 = L_0$ \\
    \For{$n=1,2,\dots$}
        {Compute $f_{n-1} = f(x_{n-1})$, $g_{n-1} = \gradB f(x_{n-1})$ \\
        Strategically update memory $\Mem = \{(x_i,f_i,g_i,\tau_i,z_{i+1}, L_i, \Delta_i)\}_{i=n-k}^{n-1}$, construct $Z$ and $G$ \\
        Select indices $s,m$ according to \eqref{Eqn:Def_idx_m} and \eqref{Eqn:Def_idx_s} \\
        \If{\eqref{Eqn:SOCP} is unbounded}{
            Set $x_n = x_m - \frac{1}{L_n} g_m$ and \textbf{break}}
        Compute $\Ucoef, \Wcoef$ by solving \eqref{Eqn:SOCP} with $\delta_n$ set according to \eqref{Eqn:deltaExpr}, and set $z' = x_0 + Z\Ucoef - G\Wcoef$, $\tau' = \ip{\Ucoef}{ \tau} + \fullSum{\Wcoef} $, and $\Delta' = \ip{\Ucoef}{\Delta}$ \\
        Set $(\tau_n,x_n,z_{n+1},\Delta_n) = \texttt{OBL\_Update}(\tau', x_m, g_m, z', L_n, \Delta', \delta_n)$ according to \cref{Alg:OBLUpdate} \\
        \uIf{$\tilde{L}_B(x_m,x_n) > L_n$}
            {
             $\quad L_{n+1} = \max\{\tilde{L}_B(x_m,x_n), 2 L_n\}, \quad \tau_n = 0, \quad z_{n+1} = x_0, \quad  \Delta_n = 0$
            }
        \ElseIf{\textit{Stopping condition}}{
            \textbf{break}
        }
        }
    \Output{ $x_n$, $\Mem$}
\end{algorithm}

\begin{remark}
    To guarantee \cref{Alg:BaseAlg} is well-defined, it is necessary to ensure that $\idxSetSuccess \neq \emptyset$. This is the role of our ``strategic update'' to our memory $\Mem$. In the case that $\tau_{n-k-1} > 0$ but $\tau_i = 0$ for $i \in [n-k, n-1]$, we discard from our memory the data indexed by $n-k$ rather than $n-k-1$ (resulting in a minor abuse of notation). In this rare scenario, this step prevents the last nonzero $\tau_i$ from being discarded and enables our continued induction with $U_i$ without losing progress.
\end{remark}

\begin{remark}
    Recall that for $i \notin \idxSetSuccess$, we have $\tau_i = 0$ and $\Delta_i = 0$. Consequently, the corresponding coefficients $\Ucoef_i$ in \eqref{Eqn:SOCP} have no effect on the subproblem's solution. We can therefore fix $\Ucoef_i = 0$ for all $i \notin \idxSetSuccess$ to reduce the number of variables in \eqref{Eqn:SOCP} for a minor computational improvement.
\end{remark}

In the following lemma, we argue that our subproblem \eqref{Eqn:SOCP} is feasible and well-defined. We defer the proof to \cref{SubSec:FeasibilityProof_Part1} and \cref{SubSec:FeasibilityProof_Part2}.

\begin{lemma} \label{Lem:Feasibility}
    Consider any sequences $x_n, \tau_n, z_{n+1}, \Delta_n$ generated by BSPGM. Then
    \begin{enumerate}[label=(\roman*)]
        \item The problem \eqref{Eqn:SOCP} is feasible.
        \item If \eqref{Eqn:SOCP} is bounded, then it is strictly feasible and strong duality holds. Moreover, the optimal value $\tau'$ is attained and satisfies $\tau' \geq \tau_\idxLastSuccess$.
        \item If \eqref{Eqn:SOCP} is unbounded, then $x_\idxArgmin - \frac{1}{L_n} g_\idxArgmin \in \argmin f$.
    \end{enumerate}
\end{lemma}

\subsection{Convergence Guarantees and Subgame Perfection of BSPGM}\label{SubSec:BSPGM_theory}
To maintain our inductive hypothesis $U_n \geq 0$, it suffices to verify that
\begin{equation} \label{Eqn:UnInduction}
    U_n = U' + \tau' Q_{m,n}(L_n) + (\tau_n - \tau') W_{\star,n}.
\end{equation}
The result then follows by the nonnegativity of our terms, provided that $Q_{m,n}(L_n) \geq 0$. This immediately provides the following convergence guarantee at each ``serious'' iteration where $Q_{m,n}(L_n) \geq 0$. We include a detailed proof in \cref{SubSec:ConvRateProof}.

\begin{theorem} \label{Thm:ConvRate}
    Consider any sequences $x_n, \tau_n, z_{n+1}, \Delta_n$ generated by \cref{Alg:BaseAlg}. 
    Then each serious step $n\in\{0,\dots, N\}$ (that is, with $\tau_n\neq 0$) has
    \begin{align} 
        f(x_n) -\frac{1}{2L_n}\|g_n\|^2_B - f(x_\star) &\leq \frac{\frac{L_n}{2} \|x_0 - x_\star\|^2_B + \Delta_n}{ \tau_n} \qquad\qquad &&\text{if\ } n <N \label{Eqn:ConvRate}\\
        f(x_N) - f(x_\star) &\leq \frac{\frac{L_N}{2} \|x_0 - x_\star\|^2_B + \Delta_N}{ \tau_N} \qquad\qquad &&\text{if\ } n=N. \label{Eqn:ConvRateN}
    \end{align}
\end{theorem}
At each ``null'' step, where $\tau_i=0$, the value of $L_n$ at least doubles. If the function $f$ is $L$-smooth with respect to $\langle \cdot,\cdot\rangle_B$ on a convex region containing all of the iterates, then there can be at most $\log_2(L/L_0)$ null steps. A simple calculation guarantees that $\tau_n$ grows at least quadratically in the number of ``serious'' steps\footnote{This can be seen since $\hat{\tau} \geq \frac{1}{2}n^2$ implies $\hat{\tau} + \frac{1+\sqrt{1+8\hat{\tau}}}{2} \geq \frac{1}{2}(n+1)^2$}. As a result, when applied to $L$-smooth convex functions, the above guarantees decrease at the minimax optimal big-O rate of $O(L\|x_0-x_\star\|^2_B/N^2)$.

\begin{remark} \label{Rem:DeltaN}
    We argue informally that our error term $\Delta_N$ should be small relative to $\tau_N$. At each step of \cref{Alg:BaseAlg}, we have $\delta_n > 0$ if and only if $L_n > L_\idxLastSuccess$ (i.e., we took a null step). Therefore, $\Delta_N$ is a function of the number of null steps and should be proportional to $\log_2(\frac{L_N}{L_0})$. Next, since $\tau_N$ grows with order $O(N^2)$ and our gradient norms $\normB{g_i}^2$ are generally decreasing, we expect $\Delta_N$ to be negligible.
    
    Through induction, for a given $n$, we can express $\Delta_n$ as a recursive product of past $\Ucoef$ values and gradient norms, but it is difficult to meaningfully bound the value of $\Delta_n$. However, a critical benefit to our dynamically optimized update steps is that $\Delta_n$ is dynamic as well. While the error term in OBL is monotonically nondecreasing with each iteration, our $\Delta_n$ is not necessarily monotone. In practice, it is not uncommon for \eqref{Eqn:SOCP} to place heavy weight in the $\Wcoef$ variables and minimal weight on the $\Ucoef$ variables at a particular iteration. This effectively resets the accumulated error $\Delta' \approx 0$ and sets $\Delta_n \approx \delta_n$. Thus, at points in our algorithm, we can obtain nearly tight guarantees when the value of $\Delta_N$ ``resets'' (See an example in \cref{Fig:DetailPlot_LogReg}).
\end{remark}

\paragraph{Minimax Optimality and Subgame Perfection without Null Steps.} Stronger theoretical guarantees can be stated for BSPGM if one restricts to first-order oracles where no null steps occur. Formally, denote a first-order oracle by a mapping $O \colon \{(x_i,f_i,g_i)\}_{i=0}^{n-1}\times \{x_n\}\mapsto (f_n,g_n)$ sending each possible $n$th iterate to a response and $O \colon \{(x_i,f_i,g_i)\}_{i=0}^{N} \mapsto (x_\star,f_\star,g_\star)$ with $g_\star=0$. Let $\mathcal{O}_{L}$ denote the set of all oracles satisfying $Q_{i,j}(L)\geq 0$ holds for every $i<j$ and $W_{\star, i} \geq 0$ for every $i$. The use of such oracles follows the ideas of~\cite{OBL}, where they considered oracles only required to satisfy a subset of the typical inequalities in smooth convex optimization.

In such settings, the dynamic optimization of BSPGM's induction ensures that its $\tau_n$ sequence grows at each iteration at least as fast as the sequence generated by the static OBL update. Then one can generate a worst-case final convergence rate for BSPGM at any iteration $n$ by considering the final $\tau_N$ produced by $N-n$ subsequent iterations of OBL (instead of BSPGM), continuing the induction. Denote the hypothetical sequence of $\tau$ produced by this procedure by $\tau_{n,n}=\tau_n$ and $\tau_{n,i}$ for $i>n$ defined recursively via~\eqref{Eqn:StandardOBLUpdate}. Hence, we have the following sequence of dynamic upper bounds on BSPGM's performance. The proof is deferred to \cref{SubSec:DynamicProof}. 

\begin{theorem} \label{Thm:DynamicGuarantees}
    Consider any oracle $O\in\mathcal{O}_L$ and let $\{(x_i,f_i,g_i,\tau_i,z_{i+1})\}_{i=0}^N$ denote the results of running BSPGM with $L_0=L$. Then one has the sequence of dynamic guarantees
    \begin{equation*}
    \frac{f_N - f_\star}{\frac{L}{2}\|x_0-x_\star\|^2_B} \leq \frac{1}{\tau_{N,N}} \leq \frac{1}{\tau_{N-1,N}} \leq \dots \leq \frac{1}{\tau_{1,N}} \leq \frac{1}{\tau_{0,N}} = \frac{2}{N(N+1) + \sqrt{2N(N+1)} }.
    \end{equation*}
\end{theorem}

The final bound above is minimax optimal among all gradient-span methods for minimization against such an oracle and is also attained by the simpler OBL method, see~\cite[Theorem 4]{OBL}. That is, letting $\mathcal{A}$ denote the set of gradient span methods producing a final iterate after $N$ iterations,
\begin{equation*}
    \min_{A\in\mathcal{A}} \max_{O\in\mathcal{O}_L} \frac{f_N - f_\star}{\frac{L}{2}\|x_0-x_\star\|^2_B} = \frac{2}{N(N+1) + \sqrt{2N(N+1)} }.
\end{equation*}
Beyond this minimax optimality, BSPGM is subgame perfect against such oracles. Given a history $\Mem = \{(x_i,f_i,g_i,\tau_i,z_{i+1})\}_{i=0}^{n-1}$ produced by BSPGM, let $\mathcal{O}^\mathcal{H}_L$ denote the subclass of oracles in $\mathcal{O}_L$ that agree with the history produced by the first $n-1$ iterations and $\mathcal{A}^\mathcal{H}$ denote the subset of algorithms producing the same first $n-1$ iterations. For any observed intermediate history, BSPGM provides a minimax optimal algorithm for the remaining subgame. The proof is deferred to \cref{SubSec:SubgamePerfectProof}.
\begin{theorem} \label{Thm:SubgamePerfect}
    For any history $\mathcal{H} = \{(x_i,f_i,g_i,\tau_i,z_{i+1})\}_{i=0}^{n-1}$ observed by BSPGM with $L=L_0$ and memory $k\geq n$ up to iteration $n\leq N$, we have
    $$ \min_{A\in\mathcal{A}^\mathcal{H}} \max_{O\in\mathcal{O}^\mathcal{H}_L} \frac{f_N - f_\star}{\frac{L}{2}\|x_0-x_\star\|^2_B} = \frac{1}{\tau_{n,N}}. $$
\end{theorem}

\subsection{Adaptive Restart} \label{SubSec:AdRestart}
Restarting is a well-known effective strategy for attaining linear rates for smooth, strongly convex problems \cite{Restart_Nemirovskii,AdaptiveRestart_ODonoghue}.  We apply this approach to improve the performance of our ultimate ASPGM, adaptively restarting an inner loop of the previously introduced BSPGM in \cref{Alg:BaseAlg}. Provided a strong convexity constant $\mu$ with respect to $\langle \cdot,\cdot \rangle_B$ is known, the dynamic induction underlying BSPGM provides a provably-good dynamic restarting condition, stated below. After developing this restart condition, we briefly introduce a heuristic for estimating $\mu$ with respect to $B$ used in our final, parameter-free method, similar in spirit to the approach of~\cite{NAGFree_Cavalcanti}.

Going forward, we will refer to each inner loop of \cref{Alg:BaseAlg} between restarts as an epoch, and we will use $\epoch{x_i}$, $\epoch{g_i}$, etc. to refer to the values of each variable at the $i$th iteration of the $\idxEpoch$th epoch.
Suppose for the sake of our theoretical development that in epoch $\ell$, a strong convexity constant $\epoch{\mu}>0$ for $f$ with respect to $\langle \cdot,\cdot\rangle_{\epoch{B}}$ is known.
Our restart condition to be checked at each serious step is
\begin{equation} \label{Eqn:RestartCondition}
    \tau_n \geq \frac{2 \epoch{L}_n}{\epoch{\mu}} + \frac{2 \epoch{\Delta}_n}{f(\epoch{x}_0) - f(\epoch{x}_n)}, \qquad\qquad f(\epoch{x}_0) - f(\epoch{x}_n) >0. 
\end{equation}
Once \eqref{Eqn:RestartCondition} is satisfied, the algorithm restarts after the next ``serious step'', taken using the final step version of the OBL update (\cref{Alg:OBLUpdate}). To restart our method, we initialize a new instance of \cref{Alg:BaseAlg} at the current iterate $\epoch{x}_n$, with $x_0^{(\idxEpoch+1)} \leftarrow \epoch{x}_n$ and some $\epoch{L}_0$ and $\epoch{B}$. 

Whenever \eqref{Eqn:RestartCondition} is satisfied by the inner loop executing \cref{Alg:BaseAlg}, the objective gap will have been reduced by at least a factor of two. Hence, by restarting the algorithm at this point (technically, at the next serious step), only a logarithmic number of restarts will be required to reach any given target accuracy. The following theorem formally states this contraction guarantee and a constant bound on the number of iterations needed to attain~\eqref{Eqn:RestartCondition}, with proof deferred to \cref{SubSec:ConvAssumingSCProof}. 

\begin{theorem} \label{Thm:ConvergenceAssumingSC}
    Suppose that $f$ is $\mu$-strongly convex with respect to $\langle \cdot,\cdot\rangle_B$. Further suppose that at some iteration $n$ of \cref{Alg:BaseAlg}, the restart condition \eqref{Eqn:RestartCondition} is satisfied. If $\epoch{x}_{n+1}$ is a serious step and $\epoch{\Delta}_{n+1} \leq \epoch{\Delta}_n$, then
    \begin{equation} \label{Eqn:AdaptiveConvergenceRate}
        f(\epoch{x}_{n+1}) - f(x_\star)\leq \frac{1}{2}\left(f(\epoch{x}_0) - f(x_\star)\right).
    \end{equation}
    Further, the restart condition~\eqref{Eqn:RestartCondition} must be attained whenever
    \begin{equation*}
    n \geq \sqrt{\frac{4 \epoch{L}_n}{\mu} + \frac{4 \epoch{\Delta}_n}{f(\epoch{x}_0) - f(x_\star)} } + \log_2\left(\frac{\epoch{L}_n}{\epoch{L_0}}\right).
    \end{equation*}
\end{theorem}

\begin{remark}
    In the case that $\epoch{\Delta}_{n+1} > \epoch{\Delta}_n$ or the next serious step does not occur until $\epoch{x}_{n+r}$ for some $r>1$, we cannot guarantee the exact contraction in \eqref{Eqn:AdaptiveConvergenceRate}, since $\epoch{\tau}_{n+r}$ may no longer satisfy \eqref{Eqn:RestartCondition} due to increases in $\epoch{\Delta}_{n+r}$ and $\epoch{L}_{n+r}$. However, in such cases, we expect a contraction approximately matching \eqref{Eqn:AdaptiveConvergenceRate}, which was sufficient in our numerical evaluations.
\end{remark}

In practice, accurately estimating $\epoch{\mu}$ (at each epoch) can be difficult and lead algorithms to be overly conservative. Numerically, we found strong practical performance by constructing an online estimate of $\mu$ similar to the approach of~\cite{NAGFree_Cavalcanti}. We initially set $\epoch{\mu}_0=\infty$ and update at each iteration according to
\begin{equation} \label{Eqn:MuCalculation}
    \epoch{\mu}_n = \min\{ \epoch{\mu}_{n-1}, \ \tilde{\mu}_{\epoch{B}}(\epoch{x}_m, \epoch{x}_n) \}.
\end{equation}
Then our ASPGM algorithm triggers a restart whenever~\eqref{Eqn:RestartCondition} holds with $\epoch{\mu}=\epoch{\mu}_n$.

\begin{remark}
    As additional heuristics, our implementation requires that each epoch perform a minimum number of iterations before restarting and forces a restart after a maximum number of iterations. This can prevent premature or delayed restarts due to a poor initial estimate of $\epoch{\mu}$ or $\epoch{L}$, respectively. In practice, we found that with a minimum iteration requirement of $20$ and maximum iteration setting of 100, performance was slightly improved.
\end{remark}

\subsection{Preconditioning} \label{SubSec:Precon}

In quasi-Newton methods, one uses an inverse Hessian approximation as a preconditioner to a gradient descent step. In L-BFGS, given a memory size $t$ of points $x_i$ and corresponding gradients $g_i$, the inverse Hessian approximation and Hessian approximation are formed recursively as follows. Letting $B_1 = B_1^{-1} = I$, the BFGS update sets
\begin{equation} \label{Eqn:BFGS_Formula}
    B_{i+1} = (I - \frac{s_i y_i^T}{y_i^T s_i}) B_i (I-\frac{y_i s_i^T}{y_i^T s_i}) +  \frac{s_i s_i^T}{y_i^T s_i},
    \qquad\qquad
    B_{i+1}^{-1} = B_i^{-1} - \frac{B_i^{-1} s_i s_i^T B_i^{-1}}{s_i^T B_i^{-1} s_i} + \frac{y_i y_i^T}{y_i^T s_i}
\end{equation}
where $s_i = x_{i+n-t} - x_{i+n-t-1}$, $y_i = g_{i+n-t} - g_{i+n-t-1}$ for $i=1,\dots,t$. We then define $B = B_{t+1}$ and $B^{-1} = B_{t+1}^{-1}$.

In ASPGM, we use the inverse Hessian approximation, derived from L-BFGS, to precondition our problem. At each restart, we ``construct'' our preconditioner $B$ by saving the last $t$ differences $s_i= x_{i+n-t} - x_{i+n-t-1}$ and $y_i = g_{i+n-t} - g_{i+n-t-1}$ into a separate memory storage, $\{(s_i,y_i)\}_{i = 1}^{t}$. We then define $B$ according to the induction \eqref{Eqn:BFGS_Formula}. With this formula, one can efficiently compute $B v$ and $B^{-1} v$ for any vector $v$, using only vector-vector products; we therefore never have to store $B$ or $B^{-1}$ explicitly. For completeness, we include \cref{Alg:InvHessianMult} and \cref{Alg:HessianMult} in \cref{Sec:Appendix} to show the standard efficient computation of $B v$ and $B^{-1} v$ (both products will be needed since $\langle x,y\rangle_B = \langle x, B^{-1} y \rangle$ and $\gradB f(x) = B \nabla f(x)$).

Finally, we can present our complete algorithm.

\begin{algorithm}[h]
\caption{ASPGM} \label{Alg:ASPGM}
    \SetKwInOut{Input}{Input}
    \SetKwInOut{Output}{Output}
    \Input{ Convex, locally smooth function $f$,  subgame memory size $k \in \mathbb{N}$,\\ preconditioning memory size $t \in \mathbb{N}$, $x_0 \in \R^d$}
    Set $B^{(1)} = I$, $x_0^{(1)} = x_0$ \\
    \For{$\idxEpoch = 1,\dots$}
        {$y = \epoch{x}_0 - 10^{-4} \cdot \frac{\nabla f(\epoch{x}_0)}{\|\nabla f(\epoch{x}_0)\|}$ \\
        $\epoch{L}_0 = \tilde{L}_\epoch{B}(\epoch{x}_0, y)$ \\ 
        $\epoch{x}_n, \Mem = \text{Output of \cref{Alg:BaseAlg}}$ given $k,\epoch{x}_0,\epoch{L}_0,\epoch{B}$\\
        Construct $B^{(\idxEpoch+1)}$ based on the last $t$ steps in $\Mem$ via~\eqref{Eqn:BFGS_Formula} \\
        $x_0^{(\idxEpoch+1)} = \epoch{x}_n$ \\
        }
\end{algorithm}

\begin{remark}
    In total, ASPGM requires the storage of  $3k + 2t$ additional vectors of size $d$ (beyond $x_0$, $x_n$, and $g_n$), where $k$ is the memory size used in the BSPGM subalgorithm, and $t$ is the memory size used for preconditioning. All other storage requirements are negligible at high dimensions. In comparison, the storage requirements of SPGM and L-BFGS are $2k$ and $2t$, respectively.
\end{remark}

    \section{Computational Benchmarks} \label{Sec:Numerics}

In this section, we demonstrate the performance of BSPGM and ASPGM across a wide range of problem classes. In both oracle complexity and wall-clock performance, our methods using modest memory sizes reach state-of-the-art performance in comparison with adaptive first-order methods and quasi-Newton methods.
In \cref{SubSec:AlgDiscussion}, we discuss our implemented versions of the algorithm, along with a large sample of competing algorithms from the literature. In \cref{SubSec:Synth}, we test ASPGM on a set of synthetic, randomly generated problems. Then in \cref{SubSec:RealRegression} and \cref{SubSec:RealOther}, we apply ASPGM to problems derived from real data sets. Finally, in \cref{SubSec:HardProbs}, we examine the performance of ASPGM on several difficult, poorly conditioned problems.

ASPGM is implemented in Julia \cite{julia} and uses MOSEK \cite{mosek} via JuMP \cite{JuMP} to solve the subproblem \eqref{Eqn:SOCP} at each iteration. Our implementation along with code used to run our experiments is available at 
\begin{equation*}
\href{https://github.com/alanluner/ASPGM}{\texttt{github.com/alanluner/ASPGM}.}
\end{equation*}
Experiments were run locally on an Intel i7 processor with 64GB of memory.

\subsection{Algorithm Varieties and Benchmarks} \label{SubSec:AlgDiscussion}

We investigate the performance of three versions of our subgame perfect methods. First, we consider the full method, ASPGM (\cref{Alg:ASPGM}), including adaptive restarting and preconditioning both with fixed memory sizes of $k=t=1$ (\texttt{ASPGM-1-1}) and $k=t=5$ (\texttt{ASPGM-5-5}). Additionally, we consider the core method BSPGM (\cref{Alg:BaseAlg}), with memory $k=7$ (\texttt{BSPGM-7}), which does not implement any adaptive restart or preconditioning.

At the end of our section, we compare in \cref{Fig:MemComparison} the performance of ASPGM with different memory sizes. As expected, we see diminishing returns as the memory size increases. This is consistent with many memory-based methods including SPGM and L-BFGS. These results justify our use of \texttt{ASPGM-5-5} as our default method. At memory size 1, the subproblem \eqref{Eqn:SOCP} can be solved explicitly, eliminating the need for an SOCP solve. This enabled us to provide a stand-alone implementation of \texttt{ASPGM-1-1} in our repository above, with no dependencies on external solvers. Our results show that even with very low memory, \texttt{ASPGM-1-1} generally maintains performance, with wall-clock performance notably benefiting from the highly efficient explicit solve.

\paragraph{Benchmark Algorithms}
We compare the performance of our base BSPGM method and our full ASPGM method with a wide range of competing algorithms. We focus primarily on adaptive methods and thus exclude methods requiring smoothness estimates, like OGM and SPGM, from our survey. Many of these methods require an unspecified initial estimate of $L_0$; we initialize $L_0 = \tilde{L}_I(x_0, x_0 - 10^{-4} \cdot \frac{\nabla f(x_0)}{\|\nabla f(x_0)\|})$. L-BFGS is implemented in Julia using the \texttt{Optim} \cite{Optim} package. All other algorithms are implemented directly in Julia.
\begin{itemize}
    \item \texttt{OBL:} Optimized Backtracking Linesearch \cite{OBL}.
    \item \texttt{UFGM:} Universal Fast Gradient Method \cite{UFGM_Nesterov}.
    \item \texttt{AdaNAG:} AdaNAG-G\textsubscript{12} from \cite[Corollary 7]{AdaNAG_Suh}.
    \item \texttt{ACFGM:} Auto-Conditioned Fast Gradient Method \cite{ACFGM_Li}. We follow the implementation of the authors and set parameters $\eta_1 = 5/(2L_0)$, $\beta=1-\sqrt{6}/3$, and $\alpha=0.1$.
    \item \texttt{AdGD:} Adaptive Gradient Descent \cite{AdGD_Prox2_Malitsky}. We use the improved adaptive gradient descent method from \cite[Algorithm 2]{AdGD_Prox2_Malitsky}.
    \item \texttt{NAGF:} NAG-free algorithm from \cite[Algorithm 2]{NAGFree_Cavalcanti}.
    \item \texttt{OSGMB:} OSGM-Best method from~\cite{OSGMBest}, a variant of Online Scaled Gradient Method~\cite{OnlineScaling_Gao}. We follow the implementation of the authors and set parameters $\omega=0$, $\tau=\frac{1}{2}L_0^2$, $\beta_0 = 0.95$, $\eta_P = \frac{1}{L_0}$, and $\eta_\beta = \min(1.0,L_0)$.
    \item \texttt{LBFGS-BL:} L-BFGS with backtracking linesearch. We use memory size $t=10$.
    \item \texttt{LBFGS-HZ:} L-BFGS with the robust linesearch introduced by Hager and Zhang \cite{HagerZhang}. Note that this is the default option when using L-BFGS in the \texttt{Optim} Julia package. We use memory size $t=10$.
\end{itemize}
Among these, the first seven provide comparisons with adaptive methods that possess strong theoretical support. Our \texttt{ASPGM-1-1} is the most comparable to these in computation/storage per iteration. The latter two L-BFGS methods are state-of-the-art in performance (although lacking non-asymptotic theory) and have similar computational/storage costs to \texttt{ASPGM-5-5} and \texttt{BSPGM-7}. Hence they set the baseline we seek to match.

We employ an oracle model in which function and gradient information are returned simultaneously (i.e., $x \mapsto (f(x), \nabla f(x) )$). Each such evaluation is considered an oracle call in our plots below. For algorithms that consider function oracles and gradient oracles separately, we apply special handling to efficiently leverage the joint first-order oracle.

\subsection{Synthetic Smooth Convex Problems} \label{SubSec:Synth}

We generate random problem instances for several problem classes discussed below, all parameterized by $A \in \R^{p \times d}$, $b \in \R^p$, $c \in \R^p$, and $x_0 \in \R^d$.
For each problem instance we set $p = 4d$, $x_0 = 0$, generate $b$ element-wise from $\mathcal{N}(0,1)$, and generate $c$ element-wise from $\mathcal{U}\{0,1\}$. Since the conditioning of each problem class is determined by the condition number of $A^T A$, we construct $A$ so that $A^T A$ has approximate condition numbers of $\kappa = 10^2$ and $\kappa = 10^4$. We also distribute the singular values of $A$ in two different ways:
\begin{itemize}
    \item Uniform distribution: $\sigma_i \sim \mathcal{U}(1,\sqrt{\kappa})$ for $i=1,\dots, d$
    \item Bimodal distribution: $\sigma_1, \dots, \sigma_{9d/10} \sim \mathcal{U}(1,1.1)$ and $\sigma_{9d/10+1},\dots, \sigma_d \sim \mathcal{U}(0.9 \sqrt{\kappa}, \sqrt{\kappa}).$
\end{itemize}
We generate two instances for each condition number and each distribution above; using dimensions $d=1000,2000,4000,8000$ and 6 problem classes below, this yields 192 synthetic problem instances.

\begin{figure}[h]
    \centering
    \includegraphics[width=1.0\textwidth]{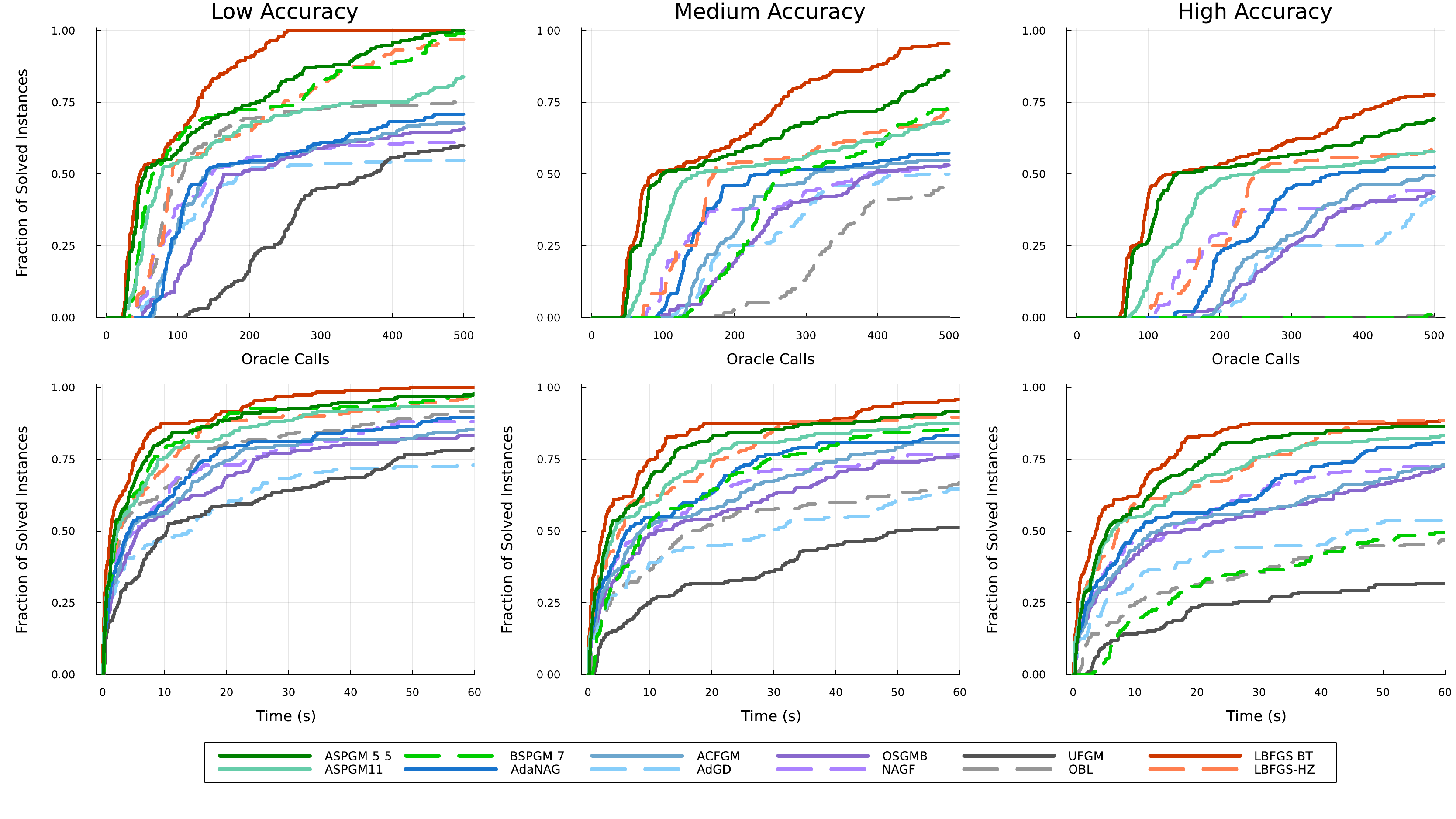}
    \caption{Performance results over synthetic problem instances, in terms of oracle calls and wall clock time. The performance is measured by $(f(x_n) - f_\star)/(f(x_0) - f_\star)$ and the target accuracies from left to right are $10^{-4}, 10^{-7}, 10^{-10}$.}
    \label{Fig:PerformancePlot_Synth}
\end{figure}

In \cref{Fig:PerformancePlot_Synth}, we present the overall performance of ASPGM, BSPGM, and competing methods across all of our problem classes. 
We also include in \cref{SubSec:AdditionalNumerics} these same performance results stratified by problem class. We see ASPGM with both memory size $5$ and a minimal size of $1$ is highly performant on our synthetic problem set, outperforming other adaptive methods, as well as \texttt{LBFGS-HZ}. This advantage is maintained when considering wall clock time. We also note the strong performance of BSPGM for attaining low-accuracy solutions, especially in terms of wall-clock time. This is particularly impressive since BSPGM also provides an explicit non-asymptotic convergence guarantee and corresponding stopping criterion via $L_n/\tau_n$.

We introduce each problem class below.

\paragraph{Regression}
We consider least squares regression and logistic regression with L2 regularization. This amounts to the following objective functions:
\begin{align}
    & f(x) = \frac{1}{2}\|Ax-b\|_2^2 \label{Eqn:Objective_LSReg}, \\
    & f(x) = \sum_{i=1}^p \log(1+\exp(c_i \cdot a_i^T x ) ) + \frac{1}{2p}\|x\|_2^2 \label{Eqn:Objective_LogReg}
\end{align}
where $a_i$ denotes the $i$th row of $A$.

\paragraph{Smoothed Feasibility Problems}
Next, we consider two methods of smoothing approximations of the feasibility problem seeking some $x$ with $A x - b \leq 0$. We define the following objectives,
\begin{align}
    & f(x) = \log \left( 1+ \sum_{i=1}^p \exp(a_i^Tx - b_i) \right), \label{Eqn:Objective_LogSumExp}\\
    & f(x) = \sum_{i=1}^p (a_i^T x - b_i)_+^2 \label{Eqn:Objective_PosSquared}
\end{align}
where $v_+$ denotes the positive part of $v$ i.e., $v_+ = \max\{v, 0\}$ and $a_i$ denotes the $i$th row of $A$.

\paragraph{Locally but not Globally Smooth Functions}

Lastly, we consider convex functions that are locally, but not globally smooth. We define the objective functions
\begin{align}
    & f(x) = \frac{1}{4}\|Ax-b\|_4^4, \label{Eqn:Objective_4Norm}\\
    & f(x) = \frac{1}{2}\|A x\|_2^2 + b^T x + \frac{1}{6p}\|x\|_2^3. \label{Eqn:Objective_CubicReg}
\end{align}

\subsection{Real-Data Regression Problems} \label{SubSec:RealRegression}

We also present performance results for problem instances on real data. Using the LIBSVM \cite{LIBSVM} library, we apply least squares regression \eqref{Eqn:Objective_LSReg} to the data sets \texttt{bodyfat}, \texttt{eunite2001}, \texttt{pyrim}, \texttt{triazines}, and \texttt{yearPredictionMSD} and logistic regression with L2 regularization \eqref{Eqn:Objective_LogReg} to the data sets \texttt{colon-cancer}, \texttt{duke-breast-cancer}, \texttt{gisette}, \texttt{leukemia}, and \texttt{madelon}. Results are shown in \cref{Fig:PerformancePlot_RealRegression}; for clarity we include only a subset of our methods, focusing on those with the strongest performance throughout our experiments. We see similar trends as in our synthetic problems, but see that ASPGM no longer outperforms \texttt{LBFGS-HZ}. We also notice a stronger difference for ASPGM between oracle performance and wall-clock performance. This can be partially attributed to these problems having lower dimension, thereby causing our subproblem to have a larger proportional impact on computation time.
In \cref{Fig:DetailPlot_LogReg}, we show detailed results for the \texttt{duke-breast-cancer} logistic regression problem. Here we show the modified performance measure $(f(x_n) - f_\star)/(\frac{1}{2}\|x_0 - x_\star\|^2)$, so we can directly compare the performance with the guarantees (ignoring error terms) of \texttt{BSPGM-7} and \texttt{OBL} in the center plot. We also show the value of the error term $\Delta_n$ for \texttt{BSPGM-7} and \texttt{OBL}. In particular, this illustrates our dynamic theory allowing $\Delta_n$ to briefly fluctuate and then return to nearly zero for \texttt{BSPGM-7}, unlike \texttt{OBL}.

\begin{figure}
    \centering
    \includegraphics[width=1.0\textwidth]{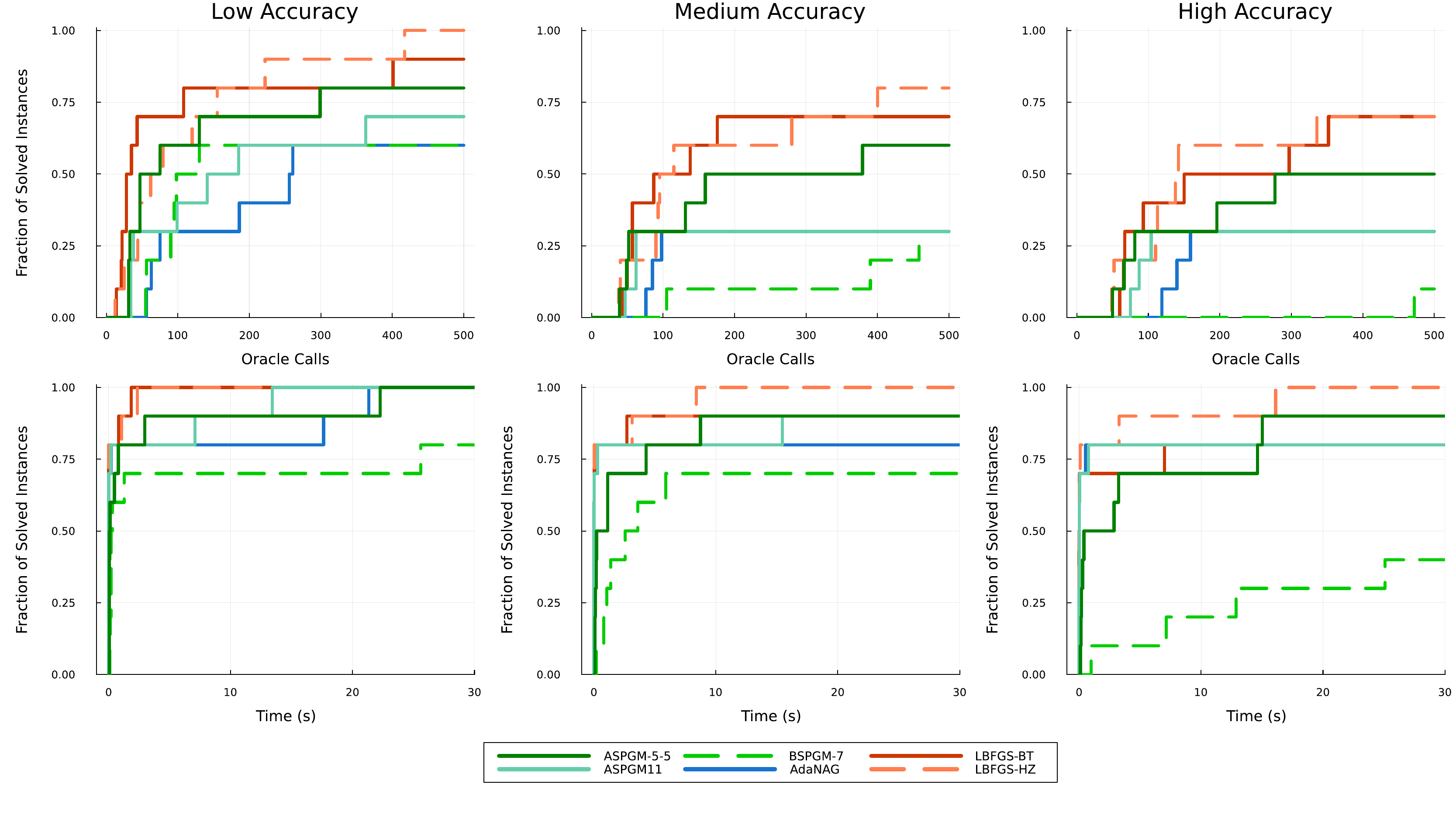}
    \caption{Performance results for least squares regression \eqref{Eqn:Objective_LSReg} and logistic regression \eqref{Eqn:Objective_LogReg} problems from LIBSVM \cite{LIBSVM}. The performance is measured by $(f(x_n) - f_\star)/(f(x_0) - f_\star)$ and the target accuracies from left to right are $10^{-4}, 10^{-7}, 10^{-10}$.}
    \label{Fig:PerformancePlot_RealRegression}
\end{figure}

\begin{figure}
    \centering
    \includegraphics[width=1.0\textwidth]{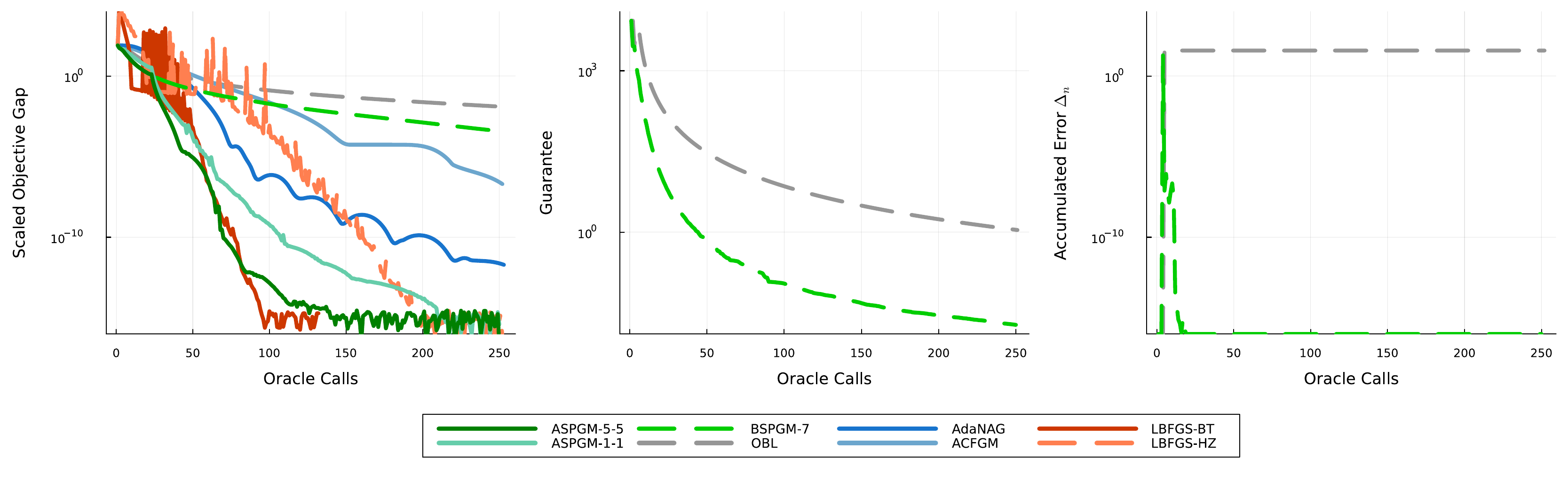}
    \caption{Performance data from logistic regression \eqref{Eqn:Objective_LogReg} run on the \texttt{colon-cancer} dataset. On the left, we show the alternate performance measure $(f(x_n) - f_\star)/(\frac{1}{2}\|x_0 - x_\star\|^2)$. In the middle plot, we compare the guarantees of \texttt{BSPGM-7} and \texttt{OBL} (ignoring error terms). This corresponds to $L_n/ \tau_n$. The right plot shows the accumulated error $\Delta_n$ for \texttt{BSPGM-7} and \texttt{OBL}.}
    \label{Fig:DetailPlot_LogReg}
\end{figure}

\subsection{Real-Data Feasibility Problems} \label{SubSec:RealOther}
As another set of problems derived from real data, we consider the search for a feasible point of an LP.  Given an LP with the constraints $Ax \leq b$, one can frame this objective as minimizing $\max_i \{(Ax-b)_i,0\}$. If we apply either of our smoothing approaches \eqref{Eqn:Objective_LogSumExp} or \eqref{Eqn:Objective_PosSquared} from \cref{SubSec:Synth}, we obtain a smooth problem suitable for our set of algorithms. We apply this method to the following problems from the ``LPfeas Benchmark'' dataset \cite{LPSource}: \texttt{brazil3}, \texttt{chromaticindex1024-7}, \texttt{ex10}, \texttt{graph40-40}, \texttt{qap15}, \texttt{savsched1}, \texttt{scpm1}, \texttt{set-cover-model}, and \texttt{supportcase10}. These include high-dimensional ($d\approx 10^5$) sparse datasets. We apply both smoothing approaches to each data file, obtaining 18 total problem instances. We see similar trends in oracle-call performance but reasonable gains in wall-clock performance.

\begin{figure}
    \centering
    \includegraphics[width=1.0\textwidth]{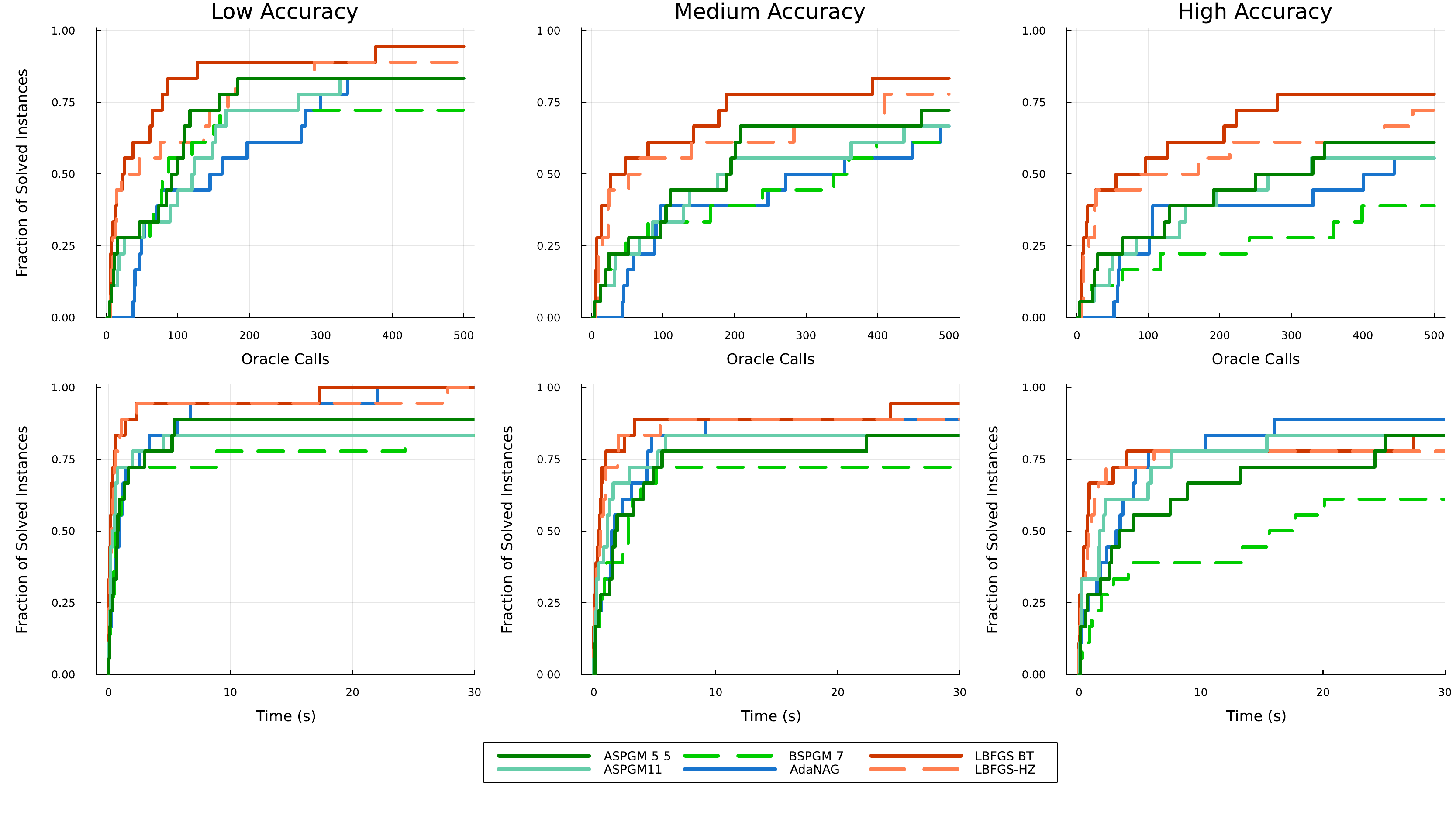}
    \caption{Performance results for feasibility problems over LP instances from \cite{LPSource}. The performance is measured by $(f(x_n) - f_\star)/(f(x_0) - f_\star)$ and the target accuracies from left to right are $10^{-4}, 10^{-7}, 10^{-10}$.}
    \label{Fig:PerformancePlot_LPFeas}
\end{figure}

\subsection{Poorly Conditioned Problems} \label{SubSec:HardProbs}

Lastly, to compare performance in highly degenerate settings, we consider several known poorly conditioned problems from the literature. Each problem is a quadratic of the form $f(x) = \frac{1}{2} x^T A x + b^T x$, with specially selected $A \in \R^{d \times d}$, $b \in \R^d$, and starting point $x_0 \in \R^d$. First, we consider the classic hard instance \cite{NesterovBook} with $x_0 = 0$, $b = (-\frac{1}{2},0,\dots,0)$ and $A$ as the tridiagonal matrix with diagonal $(1,\dots, 1)$ and super- and sub-diagonals $(-\frac{1}{2}, \dots, -\frac{1}{2})$. Next, we consider from \cite{OGMM_Florea}, the case where $b = 0$, $A$ is a diagonal matrix with entries $a_{ii} = \sin^2(\frac{\pi i}{2d})$, and $x_0 = (1/a_{11}, \dots, 1/a_{dd})$. Finally, we consider the instance with $x_0 = 0$, $b = (-1,-2,\dots,-d)$ and $A$ as the diagonal matrix with diagonal $(1^2,2^2,\dots,d^2)$. We refer to these problem instances as \texttt{Problem A}, \texttt{Problem B}, and \texttt{Problem C}, respectively.
In \cref{Fig:HardInstances}, we show our methods remain competitive in these degenerate settings with $d=1000$.

\begin{figure}
    \centering
    \includegraphics[width=1.0\textwidth]{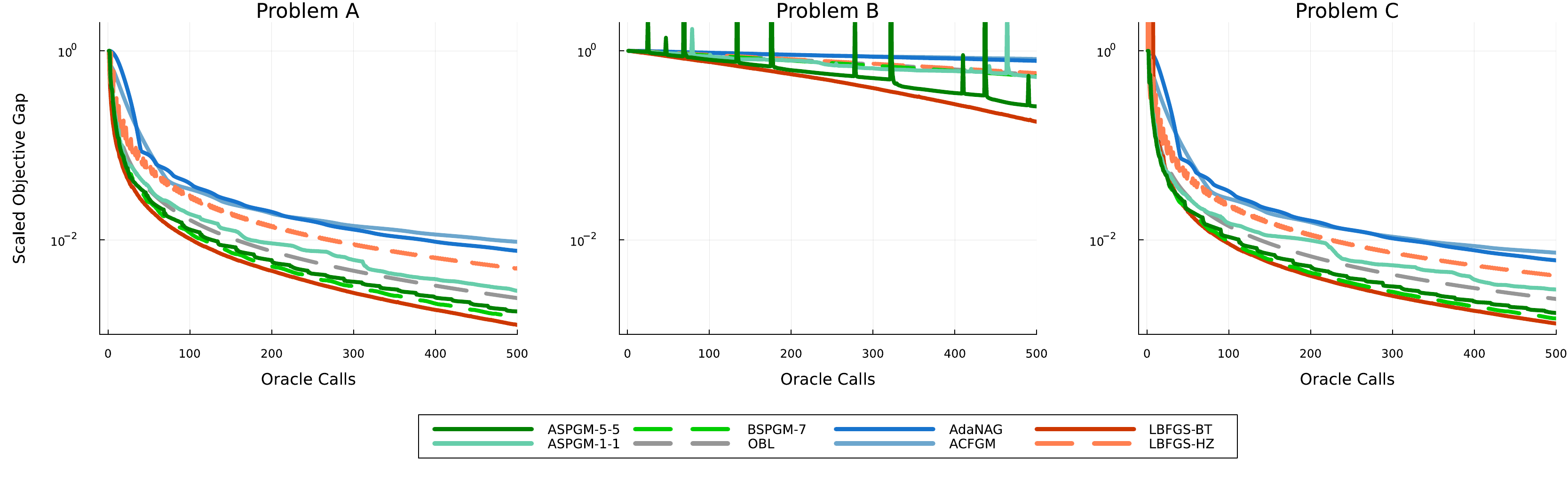}
    \caption{Performance on poorly conditioned problem instances with $d=1000$, as measured by $(f(x_n) - f_\star)/(f(x_0) - f_\star)$.}
    \label{Fig:HardInstances}
\end{figure}

\begin{figure}
    \centering
    \includegraphics[width=1.0\textwidth]{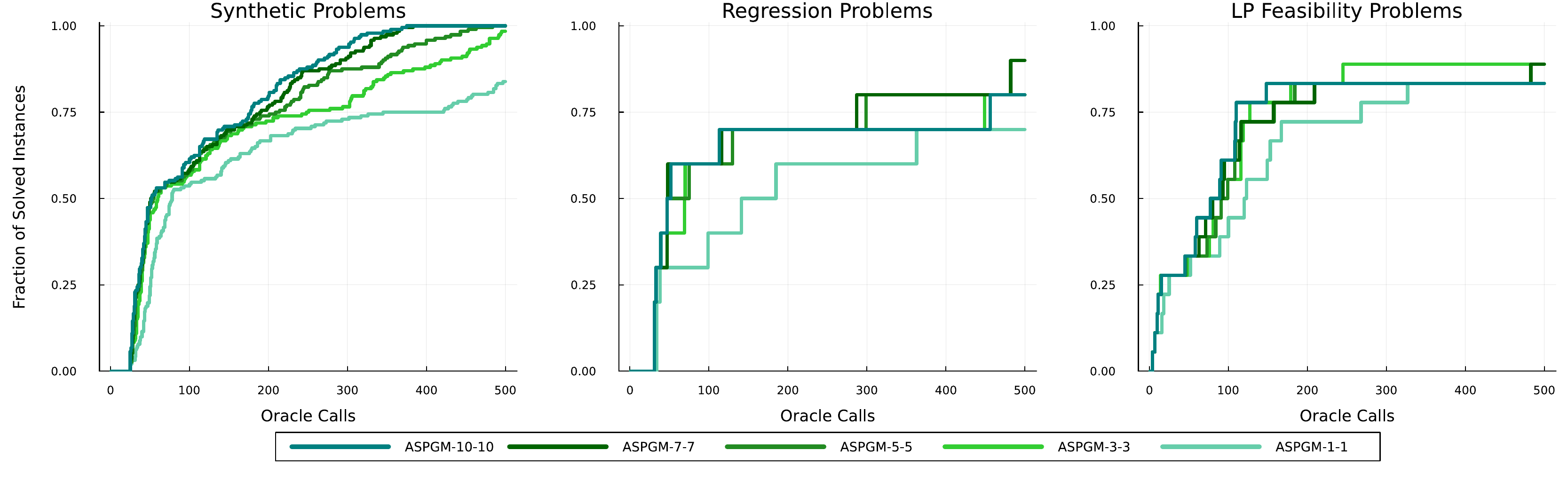}
    \caption{Performance comparison of ASPGM with different memory sizes. From left to right, we show results with performance measured by $(f(x_n) - f_\star)/(f(x_0) - f_\star)$ at accuracy $10^{-7}$, for synthetic problems, regression problems, and LP feasibility problems.}
    \label{Fig:MemComparison}
\end{figure}

    \section{Theory} \label{Sec:Theory}
This section contains proofs of our earlier theoretical results. In \cref{SubSec:FeasibilityProof_Part1} and \cref{SubSec:FeasibilityProof_Part2}, we prove \cref{Lem:Feasibility}. In \cref{SubSec:ConvRateProof}, we prove the convergence rate of our core algorithm BSPGM, stated in \cref{Thm:ConvRate}; this includes a more detailed derivation of the subproblem \eqref{Eqn:SOCP}. \cref{SubSec:DynamicProof} and \cref{SubSec:ConvAssumingSCProof} present proofs of \cref{Thm:DynamicGuarantees} and \cref{Thm:ConvergenceAssumingSC}.
To simplify notation, we will omit the subscript $B$ throughout the section on our inner products and norms, but the analysis is unchanged.

\subsection{Proof of Lemma~\ref{Lem:Feasibility} (i) and (ii)} \label{SubSec:FeasibilityProof_Part1}

Here we prove parts (i) and (ii) separately from part (iii). This is because the true sequence of logic is that (i) and (ii) imply \cref{Thm:ConvRate}, which in turn implies the result of (iii).

We first claim that if \eqref{Eqn:SOCP} is bounded, then we must have $z_{\idxLastSuccess+1} \neq x_0$ or $L_s \neq L_n$. Proving by contrapositive, if $z_{\idxLastSuccess+1} = x_0$ and $L_s = L_n$, then the solution $\rho = c e_\idxLastSuccess$, $\gamma= 0$ is feasible for all $c>0$, as
\begin{align*}
    \epsilon(c e_\idxLastSuccess, 0) &= c a_\idxLastSuccess + \delta_n - \frac{L_n}{2}\|z_{\idxLastSuccess+1} - x_0\|^2 = c \tau_\idxLastSuccess (f_\idxLastSuccess - \frac{1}{2L_\idxLastSuccess}\|g_\idxLastSuccess\|^2 - v_m) + \delta_n \\
    & = c \tau_\idxLastSuccess (v_\idxLastSuccess - v_m) + \delta_n \geq 0
\end{align*}
and thus \eqref{Eqn:SOCP} is unbounded since $\tau' \geq c \rho_\idxLastSuccess \tau_\idxLastSuccess$.

By construction of \cref{Alg:BaseAlg}, we know that $L_n \geq L_i$ for all $i \in [n-k,n-1]$. Let $\kappa = c \frac{L_n}{L_\idxLastSuccess}$ for some $c \in [0,1]$, and take $\Ucoef = \kappa e_\idxLastSuccess$ and $\Wcoef = 0$. We clearly satisfy $\Ucoef, \Wcoef \geq 0$. Then, verifying our main constraint, we have
\begin{align*}
    \epsFunc &= \ip{\Ucoef}{a} + \ip{\Wcoef}{b} +  \delta_n - \frac{L_n}{2} \norm{Z \Ucoef - G \Wcoef}^2 \\
    & = \kappa a_\idxLastSuccess + L_n \tau_\idxLastSuccess\left(\frac{1}{L_\idxLastSuccess^2} - \frac{1}{L_n^2}\right) \frac{1}{2}\norm{g_\idxLastSuccess}^2 - \frac{L_n}{2}\norm{c(z_{\idxLastSuccess+1} - x_0)}^2 \\
    & = \kappa \tau_\idxLastSuccess f_\idxLastSuccess - \kappa \tau_\idxLastSuccess \frac{1}{2L_\idxLastSuccess}\norm{g_\idxLastSuccess}^2 - \frac{\kappa L_\idxLastSuccess}{2}\norm{x_0}^2 + \frac{\kappa L_\idxLastSuccess}{2}\norm{z_{\idxLastSuccess+1}}^2 - L_n c \dotp{z_{\idxLastSuccess+1} - x_0}{x_0} \\
    & \qquad \qquad \qquad - \kappa v_\idxArgmin \tau_\idxLastSuccess + L_n \tau_\idxLastSuccess \left(\frac{1}{L_\idxLastSuccess^2} - \frac{1}{L_n^2}\right) \frac{1}{2}\norm{g_\idxLastSuccess}^2 - \frac{L_n}{2}\norm{c(z_{\idxLastSuccess+1} - x_0)}^2 \\
    & \geq \kappa \tau_\idxLastSuccess f_\idxLastSuccess -  \kappa \tau_\idxLastSuccess \frac{1}{2L_n}\norm{g_\idxLastSuccess}^2 - \kappa \tau_\idxLastSuccess v_\idxArgmin + \tau_\idxLastSuccess \left(\frac{1}{L_\idxLastSuccess} - \frac{1}{L_n} \right) \frac{1}{2}\|g_\idxLastSuccess\|^2 \\
    & \qquad \qquad \qquad + \frac{L_n (c-c^2)}{2} \|z_{\idxLastSuccess+1} - x_0\|^2 \\
    & \geq \kappa \tau_\idxLastSuccess (v_\idxLastSuccess - v_\idxArgmin) + \tau_\idxLastSuccess \left(\frac{1}{L_\idxLastSuccess} - \frac{1}{L_n} \right) \frac{1}{2}\|g_\idxLastSuccess\|^2 + \frac{L_n (c-c^2)}{2} \|z_{\idxLastSuccess+1} - x_0\|^2\\
    & \geq 0
\end{align*}
where the final inequality applies the definition of $\idxArgmin$ in \eqref{Eqn:Def_idx_m}.
This shows feasibility of $(\Ucoef, \Wcoef)$.

Suppose \eqref{Eqn:SOCP} is bounded. Then by our earlier argument either $z_{\idxLastSuccess+1} \neq x_0$ or $L_s \neq L_n$, and the expression above is strictly positive for any $c \in (0,1)$ (unless $g_s = 0$, in which case $x_s \in \argmin f$). This shows there exists a strictly feasible point for \eqref{Eqn:SOCP}, and therefore strong duality holds.
Computing the objective function with our feasible solution and $c=1$, we obtain
\begin{equation*}
    \tau' = \ip{\Ucoef}{\tau} + \fullSum{\Wcoef} = \frac{L_n}{L_\idxLastSuccess} \tau_\idxLastSuccess \geq \tau_\idxLastSuccess .
\end{equation*}
Then we have $\tau' \geq \tau_\idxLastSuccess$. Moreover, our feasible region must be compact, since all terms in our objective function are nonnegative with positive coefficients. Therefore, the optimal value $\tau'$ is attained.

\subsection{Proof of Theorem~\ref{Thm:ConvRate}} \label{SubSec:ConvRateProof}

We recall our relevant definitions:
\begin{align*}
    U_i & = \tau_i\left(f_\star - f_i + \frac{1}{2 L_i}\norm{g_i}^2\right) + \frac{L_i}{2}\norm{x_0-x_\star}^2-\frac{L_i}{2}\norm{z_{i+1}-x_\star}^2 + \Delta_i, \\
    W_{i,j} & = f_i - f_j - \dotp{g_j}{x_i - x_j}, \\
    Q_{i,j}(L) & = f_i - f_j - \dotp{g_j}{x_i - x_j} - \frac{1}{2 L} \norm{g_i - g_j}^2, \\
    v_i & = f_i - \frac{1}{2L_n}\norm{g_i}^2
\end{align*}
and we introduce the notation
\begin{align*}
    h_i & := \tau_i(f_i - \frac{1}{2L_i}\norm{g_i}^2) + \frac{L_i}{2}\norm{z_{i+1}}^2 - \frac{L_i}{2}\norm{x_0}^2, \\
    w_i & := f_i - \dotp{g_i}{x_i}.
\end{align*}
Our algorithm data $\Mem$ is stored as before, and we have the columns of $Z$ given by $\frac{L_i}{L_n}(z_{i+1} - x_0)$ and the columns of $G$ given by $\frac{g_i}{L_n}$. Additionally, we will assign to $z_{n+1}$ the form $z_{n+1} = z' - \frac{\alpha}{L_n} g_n$ for some $z' \in \R^d$ and  $\alpha \in \R$.

We first claim that $U_0 \geq 0$. Since $z_1 = x_0 - \frac{1}{L_0} g_0$, and $\tau_0 = 1$, this can be seen as
\begin{align*}
    U_0 &= \tau_0\left(f_\star - f_0 + \frac{1}{2 L_0}\norm{g_0}^2\right) + \frac{L_0}{2}\norm{x_0-x_\star}^2-\frac{L_0}{2}\norm{z_1-x_\star}^2 \\
    & = f_\star - f_0 - \dotp{g_0}{x_\star - x_0} \\
    & = W_{\star,0} \geq 0
\end{align*}
where the inequality follows from the convexity of our function $f$.

Let $\idxSetSuccess = \{i \in [n-\idxMem, n-1] \mid \tau_i > 0\}$ and recall that our strategic memory update in \cref{Alg:BaseAlg} ensures that $\idxSetSuccess \neq \emptyset$ for any iteration $n$. We can therefore set $\idxLastSuccess = \max\{i \in \idxSetSuccess\}$ and $\idxArgmin \in \argmin_{i \in \idxSetSuccess} v_i$, and then $U'$ from \eqref{Eqn:Def_UPrime} is properly defined for some $z'$ and $\tau',\Delta' \geq 0$. The unknown variables are $f_\star$, $f_n$, $g_n$, and $x_\star$, so we can view \eqref{Eqn:UPrimeInduction} in terms of a polynomial of these variables. Accordingly, we rewrite \eqref{Eqn:Def_UPrime} and \eqref{Eqn:UPrimeInduction} as
\begin{align*}
    U' &= \left(\frac{L_n}{2}\norm{x_0}^2 - \frac{L_n}{2}\norm{z'}^2 - \tau' v_\idxArgmin + \Delta' + \delta_n \right) + \tau' f_\star + \langle L_n(z'-x_0), x_\star \rangle,  \\
    U' &= \left(-\ip{\Ucoef}{h} - \ip{\Wcoef}{w} + \ip{\Ucoef}{\Delta} + \epsilon \right) + \left(\ip{\Ucoef}{\tau} + \fullSum{\Wcoef} \right) f_\star \\
    & \qquad \qquad \qquad + \langle L_n Z \Ucoef - L_n G \Wcoef, x_\star \rangle.
\end{align*}
Setting these equal, we obtain the equations
\begin{align}
    \frac{L_n}{2}\norm{x_0}^2 - \frac{L_n}{2}\norm{z'}^2 - \tau' v_\idxArgmin + \Delta' + \delta_n & = -\ip{\Ucoef}{h}  - \ip{\Wcoef}{w} + \ip{\Ucoef}{\Delta} + \epsilon , \label{Eqn:Derivation_Eq1} \\ 
    \tau' & = \ip{\Ucoef}{\tau} + \fullSum{\Wcoef} , \label{Eqn:Derivation_Eq2} \\
    L_n(z' - x_0) &= L_n Z \Ucoef - L_n G \Wcoef. \label{Eqn:Derivation_Eq3}
\end{align}

Next, we claim 
\begin{equation}
    U_n = U' + \tau' Q_{m,n}(L_n) + (\tau_n - \tau') W_{\star,n} \tag{\ref{Eqn:UnInduction}}
\end{equation}
for some value $\tau_n$. We can rewrite $U_n$ in terms of $z'$ as
\begin{align*}
    U_n &=  \left(\frac{L_n}{2}\norm{x_0}^2 - \frac{L_n}{2} \norm{z'}^2 + \Delta_n \right) + \tau_n f_\star - \tau_n f_n + \langle \alpha z', g_n \rangle + \langle L_n(z' - x_0), x_\star \rangle \\
    & \qquad \qquad \qquad + (\tau_n \indicator{n<N} - \alpha^2 ) \frac{1}{2L_n}\|g_n\|^2 -\alpha \langle g_n, x_\star \rangle
\end{align*}
where $\indicator{n<N}$ is the indicator function for $n<N$.
Then expanding terms of \eqref{Eqn:UnInduction}, we have
\begin{align}
    \Delta_n &= \Delta' + \delta_n, \label{Eqn:Derivation_Eq4} \\
    \alpha z' &= \tau' \left(x_n - (x_\idxArgmin - \frac{1}{L_n} g_\idxArgmin)\right) + (\tau_n - \tau') x_n, \label{Eqn:Derivation_Eq5} \\
    \tau_n \indicator{n<N} - \alpha^2 & = -\tau', \label{Eqn:Derivation_Eq6} \\
    \alpha & = \tau_n - \tau'. \label{Eqn:Derivation_Eq7} 
\end{align}
The conditions \eqref{Eqn:Derivation_Eq6} and \eqref{Eqn:Derivation_Eq7} exactly recover the construction of $\tau_n$ in \cref{Alg:OBLUpdate}.
We can simplify \eqref{Eqn:Derivation_Eq3} and \eqref{Eqn:Derivation_Eq5} to obtain explicit expressions for $z'$ and $x_n$:
\begin{align}
    z' &= x_0 + Z \Ucoef - G \Wcoef, \\
    x_n & = \frac{1}{\tau_n}\left(\tau'(x_\idxArgmin - \frac{1}{L_n} g_\idxArgmin) + (\tau_n - \tau') z' \right).
\end{align}

Our goal is to obtain the best inductive hypothesis $U' \geq 0$. Thus, we want to maximize $\tau'$ while keeping $\Delta'$ bounded, and subject to our constraints \eqref{Eqn:Derivation_Eq1}-\eqref{Eqn:Derivation_Eq7} and the nonnegativity of our terms $\Ucoef$, $\Wcoef$, and $\epsilon$. Recall from \cref{SubSec:FeasibilityProof_Part1}, that the choice $\Ucoef = \frac{L_n}{L_\idxLastSuccess} e_\idxLastSuccess, \Wcoef = 0$ is feasible. Moreover, this aligns with the standard induction for OBL \eqref{Eqn:StandardOBLUpdate}, for which $\Delta' = \frac{L_n}{L_\idxLastSuccess} \Delta_\idxLastSuccess = \ip{\Delta}{\Ucoef}$. Therefore, it is valid to bound\footnote{While there are several options for bounding $\Delta'$, we choose this bound because it allows $\Delta_n$ to not be monotonically increasing (See \cref{Rem:DeltaN}).}
\begin{equation} \label{Eqn:Delta_nExpr}
    \Delta' \leq \ip{\Ucoef}{\Delta}.
\end{equation}
This yields the optimization problem:
\begin{align} \label{Eqn:OptProb1}
    \left\{\begin{array}{ll}
            \max_{\Ucoef, \Wcoef, \Delta', \epsilon} & \ip{\Ucoef}{\tau} + \fullSum{\Wcoef} \\
            \text{s.t.} & \frac{L_n}{2}\norm{x_0}^2 - \frac{L_n}{2} \norm{x_0 + Z \Ucoef - G \Wcoef}^2 = \sum_{i=n-k}^{n-1} \Ucoef_i (-h_i + v_\idxArgmin \tau_i) \\
            &  \qquad \qquad \qquad \qquad  + \sum_{i=n-k}^{n-1} \Wcoef_i (-w_i + v_\idxArgmin) + \ip{ \Ucoef}{\Delta} - \Delta' - \delta_n + \epsilon \\
        & 0 \leq \Delta' \leq \ip{\Ucoef}{\Delta} \\
        & \Ucoef,\, \Wcoef, \, \epsilon \geq 0 .
    \end{array}\right.
\end{align}

We can take several steps to simplify \eqref{Eqn:OptProb1}. We can remove $\epsilon$ from our first constraint and replace it with an inequality. Next, due to the role of $\Delta'$ in the first constraint and its absence from the objective, one can see it will be pushed to its maximum value. We therefore fix $\Delta' = \ip{\Ucoef}{\Delta}$. With these simplifications, and rearranging our first constraint, we obtain
\begin{align} \label{Eqn:OptProb2}
    \left\{\begin{array}{ll}
            \max_{\Ucoef, \Wcoef} & \ip{\Ucoef}{\tau} + \fullSum{\Wcoef}\\
            \text{s.t.} & \frac{L_n}{2} \norm{Z \Ucoef - G \Wcoef}^2 \leq  \sum_{i=n-k}^{n-1} \Ucoef_i (h_i - v_m \tau_i) + \sum_{i=n-k}^{n-1} \Wcoef_i (w_i - v_m) \\
            & \qquad \qquad \qquad \qquad \qquad - \langle L_n Z \Ucoef, x_0 \rangle + \langle L_n G \Wcoef, x_0 \rangle + \delta_n \\
        & \Ucoef,\, \Wcoef \geq 0.
    \end{array}\right.
\end{align}
Recalling our definition of $\epsFunc$ from \eqref{Eqn:epsFunc}, this is exactly our optimization problem \eqref{Eqn:SOCP}.
Finally, with $\tau'$ as the solution to \eqref{Eqn:OptProb2}, i.e. $\tau' = \ip{\Ucoef}{\tau} + \fullSum{\Wcoef}$, and $\tau_n$ set according to \cref{Alg:OBLUpdate}, we recover the exact form of \cref{Alg:BaseAlg}.

All that remains to be shown is that $U_n \geq 0$. By assumption, we have that $U_i \geq 0$ for all $i \in \idxSetSuccess$ and $Q_{\idxArgmin, n}(L_n) \geq 0$. Combining with the nonnegativity of our terms $W_{\star,i}$, equations \eqref{Eqn:UPrimeInduction} and \eqref{Eqn:UnInduction} yield the fact that $U_n \geq 0$. Rearranging $U_n \geq 0$ and $U_N \geq 0$ yield the final equations in our theorem.

\subsection{Proof of Lemma~\ref{Lem:Feasibility} (iii)} \label{SubSec:FeasibilityProof_Part2}

Suppose that \eqref{Eqn:SOCP} is unbounded. Then its feasible domain must be unbounded. By the convexity of our feasible region, the nonnegativity of $\Ucoef, \Wcoef$ and the feasibility of $(\Ucoef,\Wcoef) = (0,0)$ in \eqref{Eqn:SOCP}, this means there must exist nonzero $(\Ucoef, \Wcoef)$ such that $(c\Ucoef, c\Wcoef)$ is feasible for all $c \geq 0$. Since our quadratic constraint must hold as $c \to \infty$, we must have $Z \Ucoef - G \Wcoef = 0$. Then, with $\tau' = c(\ip{\Ucoef}{\tau} + \fullSum{\Wcoef})$ and $z' = x_0 + c(Z \Ucoef - G \Wcoef) = x_0$, we apply our OBL update from \cref{Alg:BaseAlg}.
Observe that $\tau' \to \infty$ as $c$ increases since $\ip{\Ucoef}{\tau} + \fullSum{\Wcoef} > 0$. Moreover, we see that $\frac{\tau'}{\tau_n} \to 1$ as $\tau' \to \infty$. Thus, as $c \to \infty$, we get $x_n \to x_\idxArgmin - \frac{1}{L_n} g_\idxArgmin$.

Applying our guarantee from \eqref{Eqn:ConvRate}, we have
\begin{equation*}
    f(x_n) - f_\star \leq \frac{1}{\tau_n}\left( \frac{L_n}{2}\|x_0 - x_\star\|^2 + \Delta_n \right) .
\end{equation*}
Taking the limit of both sides using our above arguments, we obtain $f(x_\idxArgmin - \frac{1}{L_n} g_\idxArgmin) \leq f_\star$ and therefore  $x_\idxArgmin - \frac{1}{L_n} g_\idxArgmin \in \argmin f$.

\subsection{Proof of Theorem~\ref{Thm:DynamicGuarantees}} \label{SubSec:DynamicProof}

We will assume that at each iteration, \eqref{Eqn:SOCP} is bounded; otherwise, from \cref{Lem:Feasibility} we would have found a minimizer of $f$ and the algorithm would terminate. We show by induction that $\tau_{i,i} \geq \tau_{n,i}$ for all $i \in [n,N]$. For our base case, we have by definition $\tau_n = \tau_{i,i} = \tau_{n,i}$ at $i = n$.

Note that in the setting of $\mathcal{O}^\mathcal{H}_L$, at any iteration $n$ we have $\idxLastSuccess = n-1$.
Then from \cref{Lem:Feasibility}, we know that at iteration $i > n$, \cref{Alg:BaseAlg} will select $\tau'$ such that $\tau' \geq \tau_{i-1,i-1}$. Applying our update from \cref{Alg:BaseAlg}, we have
for $n < i < N$,
\begin{equation*}
    \tau_{i,i} =\tau' + \frac{1 + \sqrt{1+8\tau'}}{2}  \geq \tau_{i-1,i-1} + \frac{1 + \sqrt{1+8\tau_{i-1,i-1}}}{2} 
     = \tau_{i-1,i}
\end{equation*}
and similar for the case $i=N$.
Then by induction, using our standard guarantee from $U_N$, we have
\begin{equation*}
    \frac{f_N - f_\star}{\frac{L}{2}\|x_0-x_\star\|^2} \leq \frac{1}{\tau_{N,N}} \leq \frac{1}{\tau_{N-1,N}} \leq \dots \leq \frac{1}{\tau_{1,N}} \leq \frac{1}{\tau_{0,N}}.
\end{equation*}
The final equality of our theorem $\frac{1}{\tau_{0,N}} = \frac{2}{N(N+1)+\sqrt{2N(N+1)}}$ comes directly from \cite[Theorem 4]{OBL}.

\subsection{Proof of Theorem~\ref{Thm:ConvergenceAssumingSC}} \label{SubSec:ConvAssumingSCProof}

Since $\epoch{x}_{n+1}$ is a serious step, we know that $\epoch{L}_{n+1} = \epoch{L}_n$ and we assume $\epoch{\Delta}_{n+1} \leq \epoch{\Delta}_n$. Moreover, from \cref{Lem:Feasibility}, we know $\epoch{\tau}_{n+1} \geq \epoch{\tau}_n$.

We therefore have
\begin{align*}
    \epoch{\tau}_{n+1} \geq \epoch{\tau}_n & \geq \frac{2 \epoch{L}_n}{\mu} + \frac{2 \epoch{\Delta}_n}{f(\epoch{x}_0) - f(\epoch{x}_n)} \\
    & \geq \frac{2 \epoch{L}_{n+1}}{\mu} + \frac{2 \epoch{\Delta}_{n+1}}{f(\epoch{x}_0) - f(x_\star)} \\
    & = \frac{2\epoch{L}_{n+1} (f(\epoch{x}_0) - f(x_\star)) + 2 \mu \epoch{\Delta}_{n+1}}{\mu(f(\epoch{x}_0) - f(x_\star))}.
\end{align*}
Then by our convergence rate from \eqref{Eqn:ConvRateN}, having applied the ``final step'' update to obtain $\epoch{x}_{n+1}$, we can write
\begin{align*}
    f(\epoch{x}_{n+1}) - f(x_\star) & \leq \frac{1}{\epoch{\tau}_{n+1}} \left(\frac{\epoch{L}_{n+1}}{2}\|\epoch{x}_0 - x_\star\|^2 + \epoch{\Delta}_{n+1} \right)\\
    &\leq \frac{1}{\epoch{\tau}_{n+1}} \left(\frac{\epoch{L}_{n+1}}{\mu}(f(\epoch{x}_0) - f(x_\star)) + \epoch{\Delta}_{n+1}\right) \\
    & \leq \left(\frac{\epoch{L}_{n+1}(f(\epoch{x}_0) - f(x_\star)) + \mu \epoch{\Delta}_{n+1}}{\mu}\right) \\
    & \qquad \qquad \cdot \left(\frac{\mu(f(\epoch{x}_0) - f(x_\star))}{2\epoch{L}_{n+1}(f(\epoch{x}_0) - f(x_\star)) + 2 \mu \epoch{\Delta}_{n+1}} \right) \\
    & = \frac{f(\epoch{x}_0) - f(x_\star)}{2}
\end{align*}
where the second inequality comes from the strong convexity inequality \eqref{Eqn:SCIneq} applied to $x = x_\star$ and $y = \epoch{x}_0$.

Recall that there can be at most $\log_2(\frac{\epoch{L}_{n}}{\epoch{L}_0})$ null steps in epoch $\idxEpoch$ by iteration $n$. Accounting for null steps in the result of \cref{Thm:DynamicGuarantees}, we have that $\epoch{\tau}_{n} \geq \frac{1}{2}(n - \log_2(\frac{\epoch{L}_{n}}{\epoch{L}_0}) )^2$. Rearranging, we get that if our restart condition is not satisfied, then
\begin{equation*}
    n \leq \sqrt{2\epoch{\tau}_{n}} + \log_2\left(\frac{\epoch{L}_{n}}{\epoch{L}_0}\right) < \sqrt{\frac{4 \epoch{L}_n}{\mu} + \frac{4 \epoch{\Delta}_n}{f(\epoch{x}_0) - f(x_\star)} } + \log_2\left(\frac{\epoch{L}_n}{\epoch{L}_0}\right).
\end{equation*}

    \paragraph{Acknowledgments.} This work was supported in part by the Air Force Office of Scientific Research under award number FA9550-23-1-0531. Benjamin Grimmer was additionally supported as a fellow of the Alfred P. Sloan Foundation.

    {\small
    \bibliographystyle{unsrt}
    \bibliography{references}
    }

    \appendix
    \section{Appendix} \label{Sec:Appendix}

\subsection{Preconditioning Algorithms} \label{SubSec:PreconditioningAlgs}

We briefly present standard algorithms for efficient computation of matrix-vector products with the inverse Hessian approximation $B$ and the Hessian approximation $B^{-1}$. \cref{Alg:InvHessianMult} follows the standard two-loop recursion method \cite{LBFGS_Nocedal}. \cref{Alg:HessianMult} uses the compact representation \cite{ByrdNocedal_1994} of the Hessian approximation $B^{-1}$,
\begin{equation*}
    B^{-1} = \theta I - \begin{bmatrix} \theta S & Y \end{bmatrix} \begin{bmatrix} \theta S^T S & T \\ T^T & -D \end{bmatrix}^{-1} \begin{bmatrix} \theta S^T \\ Y^T \end{bmatrix}
\end{equation*}
to compute $B^{-1} v$ without storing any full-dimension matrices (see \cref{Alg:HessianMult} for variable definitions).

\begin{minipage}{0.49\textwidth}
\AlgSpacingA
\begin{algorithm}[H]
\caption{Efficient calculation of $Bv$ \cite{LBFGS_Nocedal}} \label{Alg:InvHessianMult} 
        \SetKwInOut{Input}{Input}
        \SetKwInOut{Output}{Output}
        \Input{ $\mathcal{M} := \{(s_i, y_i)\}_{i=1}^{t}$, $v \in \R^d$}
        $q = v$ \\
        \For{$i=t, t-1, \dots, 1$}
        {$\eta_i = \frac{1}{ y_i^T s_i}$ \\
        $\alpha_i = \eta_i (s_i^T q)$  \\
        $q = q - \alpha_i y_i$
        }
        $r = \frac{s_t^T  y_t }{y_t^T y_t} q$ \\
        \For{$i=1, \dots, t$}
        {$\beta_i = \eta_i (y_i^T r) $ \\
        $r = r + (\alpha_i - \beta_i) s_i$
        }
        \Output{ $r$}
\end{algorithm}
\end{minipage}
\hfill
\begin{minipage}{0.49\textwidth}
\AlgSpacingB
\begin{algorithm}[H]
\caption{Efficient calculation of $B^{-1} v$ \cite{ByrdNocedal_1994}}\label{Alg:HessianMult} 
    \SetKwInOut{Input}{Input}
    \SetKwInOut{Output}{Output}
    \Input{ $\mathcal{M} := \{(s_i, y_i)\}_{i=1}^{t}$, $v \in \R^d$}
    $\theta = \frac{y_t^T y_t}{s_t^T y_t}$ \\
    $D = \text{diag}(s_1^T y_1, \dots, s_t^T y_t)$ \\
    $S = \begin{bmatrix} s_1 & \cdots & s_t\end{bmatrix} \quad Y= \begin{bmatrix} y_1 & \cdots & y_t\end{bmatrix}$ \\
    $T_{i,j} = \begin{cases} s_i^T y_j \quad & \text{if } i>j \\ 0 & \text{otherwise} \end{cases}$ \\
    $M = \begin{bmatrix} \theta S^T S & T \\ T^T & -D \end{bmatrix}$ \\
    $q = M^{-1} \begin{bmatrix} \theta S^T v \\ Y^T v \end{bmatrix}$ \\
    $r = \theta v - \begin{bmatrix} \theta S & Y \end{bmatrix} q$ \\
    \Output{ $r$}
\end{algorithm}
\end{minipage}

\subsection{Proof of Theorem~\ref{Thm:SubgamePerfect}} \label{SubSec:SubgamePerfectProof}

From \cref{Thm:DynamicGuarantees}, we have the upper bound
\begin{equation}
    \min_{A\in \mathcal{A}^\mathcal{H}} \max_{O \in\mathcal{O}_L^\mathcal{H}} \frac{f(x_N)-f(x_\star)}{\frac{L}{2}\|x_0-x_\star\|^2} \leq \frac{1}{\tau_{n,N}}.
\end{equation}
Thus, to demonstrate the minimax optimality, it suffices to provide a matching lower bound.
To do so, we will construct a particular oracle that is hard for all $A\in \mathcal{A}^\mathcal{H}$.

First, due to our oracle class $\mathcal{O}_L^\mathcal{H}$, we know $L = L_0 = \dots, L_n$. In this setting, we immediately have $\Delta' = 0$ and $\delta_n = 0$.
Our subproblem therefore reduces to the following:
\begin{equation} \label{Eqn:SOCP_Simplified}
    \tau' = \left\{\begin{array}{ll}
            \max_{\Ucoef, \Wcoef} & \ip{\Ucoef}{\tau} + \fullSum{\Wcoef} \\
            \text{s.t.} & \ip{\Ucoef}{a} + \ip{\Wcoef}{b} - \frac{L}{2} \norm{Z \Ucoef - G \Wcoef}^2 \geq 0 \\
            & \Ucoef, \, \, \Wcoef \geq 0. \\
            \end{array} \right.
\end{equation}

Next, we introduce some notation. Let $\phi_n$ be the optimal value of \eqref{Eqn:SOCP} at iteration $n$, (i.e., $\phi_n = \ip{\Ucoef}{\tau} + \fullSum{\Wcoef}$) and define $\tau_{n,i}$ as in \cref{Thm:DynamicGuarantees}. Further define
\begin{align*}
    \phi_i &= \tau_{n,i-1} \quad \forall i \in [n+1,N] \\
        \psi_i &= \begin{cases} \frac{1 + \sqrt{1+8\phi_i}}{2} & \text{if } n \leq i \leq N-1 \\ \sqrt{\phi_i} & \text{if } i=N
    \end{cases} \\
    \tau_{n,i} &= \phi_i + \psi_i.
\end{align*}

We now consider the dual program to \eqref{Eqn:SOCP_Simplified} and obtain the following lemma. Recall from \cref{Lem:Feasibility} that a bounded solution to \eqref{Eqn:SOCP} implies that strong duality holds. The derivation of the dual problem below is effectively identical to that in~\cite[Lemma 3]{SPGM}.

\begin{lemma}
    The dual program to \eqref{Eqn:SOCP_Simplified} is
    \begin{equation}\label{Eqn:DualSOCP}
        \left\{\begin{array}{ll}
        \inf_{\xi,y} & \frac{2\|y\|^2}{\xi L} + \frac{\xi}{2}\|x_0\|^2 + 2 \langle x_0, y \rangle \\
        \text{\emph{s.t.}} & \xi(-h + v_m \tau) - 2 Z^T y \geq \tau \\
        & \xi(-w + v_m \ones) + 2G^T y \geq 1 \\
        & \xi > 0
        \end{array} \right.
    \end{equation}
    Strong duality holds between \eqref{Eqn:SOCP_Simplified} and \eqref{Eqn:DualSOCP}. If $(\Ucoef, \Wcoef)$ and $(\xi, y)$ are optimal solutions to \eqref{Eqn:SOCP} and \eqref{Eqn:DualSOCP} respectively, then
    \begin{equation*}
        \frac{-2y}{\xi L} = x_0 + Z \Ucoef - G \Wcoef = z'.
    \end{equation*}
\end{lemma}

We continue following the approach of \cite{SPGM} and define $f_\star = v_m - \frac{1}{\xi}$. Then our feasibility constraints for $\xi$ can be rewritten for all $i \in [n-k,n-1]$ as
\begin{align}
    \tau_i\left(f_\star - f_i + \frac{1}{2L}\|g_i\|^2 \right) &+ \frac{L}{2}\|x_0\|^2 - \frac{L}{2}\|z_{i+1}\|^2 + L\langle z_{i+1} - x_0, z' \rangle \geq 0 \label{Eqn:DualFeas1} \\
    &f_\star - f_i - \langle g_i, z' - x_i \rangle \geq 0 . \label{Eqn:DualFeas2}
\end{align}
Letting $\sigma = \|z' - x_0\|^2$ and applying strong duality we have
\begin{align*}
    \phi_n &= \frac{2\|y\|^2}{\xi L} + \frac{\xi L}{2}\|x_0\|^2 + 2 \langle x_0, y \rangle \\
    & = \frac{\xi L}{2} \left\|\frac{2y}{\xi L} + x_0 \right\|^2 \\
    & = \frac{\xi L \sigma}{2}
\end{align*}
from which we obtain $f_\star = v_m - \frac{L \sigma}{2\phi_n}$.

Let $e_0, e_1, \dots$ be an orthonormal basis. We assume without loss of generality that $g_i \in \text{span}(e_0, \dots, e_i)$ for all $i < n$. We now define the responses of our oracle $x_i \mapsto (f_i, g_i)$ for $i = n, n+1, \dots, N$ by
\begin{align}
    f_i &= f_\star + L \sigma \psi_i \beta_i = v_m + \frac{L \sigma}{2} (2 \psi_i \beta_i - \frac{1}{\phi_n}), \\
    g_i &= L\sqrt{\sigma \beta_i} e_i 
\end{align}
where $\beta_i$ is defined inductively as $\beta_i = \frac{1}{\phi_i(2 \psi_i + 1)} \left( 1 + \sum_{j=n}^{i-1} \psi_j^2 \beta_j\right)$ with $\beta_n = \frac{1}{\phi_n (2\psi_n + 1)}$. Note that this oracle selection is independent of the value of $x_i$; this is permissible since our oracle is not restricted to return values corresponding to an actual function\footnote{Indeed, this construction even allows the oracle to be ``inconsistent'' over time: If an algorithm selects $x_1 = x_0$, the oracle can return $(f_1,g_1)$ with $f_1 \neq f_0$ and $g_1 \neq g_0$, as long as $Q_{0,1}(L) \geq 0$ still holds.}. It is only required to satisfy $\{Q_{i,j}\}_{i < j < n} \cup \{W_{\star,i}\}_{i=0}^{n-1}$. Lastly, we set $x_\star = z_{N+1}$.

We claim that this oracle does indeed satisfy our inequality set $\{Q_{i,j}\}_{i < j < n} \cup \{W_{\star,i}\}_{i=0}^{n-1}$. We will verify each inequality directly, but to do so, we first need a few useful lemmas.

\begin{lemma} \label{Lem:BetaInequalities}
    Let $j > i \geq n$. Then the following hold
    \begin{align}
        \beta_i &\leq \frac{1}{\phi_n (2\psi_i + 1)} \label{Eqn:BetaLemma1} \\
        (2\psi_j + 1) \beta_j & \leq (2 \psi_i - 1) \beta_i. \label{Eqn:BetaLemma2}
    \end{align}
\end{lemma}

\begin{proof}
    Note from our definition of $\psi_i$ that $\psi_i^2 = 2 \tau_{n,i} - \psi_i = \phi_i + \tau_{n,i}$ for $n \leq i < N$ and $\psi_N^2 = \tau_{n,N} - \psi_N = \phi_N$. From this we have
    \begin{equation*}
        \tau_{n,i}(2\psi_i - 1) = \begin{cases}
            \phi_i(2\psi_i+1)+\psi_i^2 \quad & \text{if } i \in [n, N-1] \\
            \phi_N(2\psi_N+1)+\psi_N^2 - \tau_{n,N} \quad & \text{if } i = N.
        \end{cases}
    \end{equation*}
Using this relation, we can write
\begin{align*}
    \phi_{i+1}(2 \psi_{i+1} + 1) \beta_{i+1} &= 1 + \sum_{\ell = n}^i \psi_\ell^2 \beta_\ell \\
    & = 1 + \sum_{\ell = n}^{i-1} \psi_\ell^2 \beta_\ell + \psi_i^2 \beta_i \\
    & = \phi_{i}(2 \psi_{i} + 1) \beta_{i} + \psi_i^2 \beta_i \\
    & = \tau_{n,i}(2 \psi_i - 1) \beta_i.
\end{align*}
Recognizing that $\phi_{i+1} = \tau_{n,i}$, we rearrange to obtain the identity 
\begin{equation} \label{Eqn:BetaIdentity}
    (2 \psi_{i+1} + 1) \beta_{i+1} = (2 \psi_i -1) \beta_i.
\end{equation}

We now prove \eqref{Eqn:BetaLemma1} by induction. By definition, $\beta_n = \frac{1}{\phi_n(2 \psi_n+1)}$, so our base case holds. Then for $i>n$, using \eqref{Eqn:BetaIdentity} we can write
\begin{align*}
    \beta_i = \frac{(2 \psi_{i-1} -1)\beta_{i-1}}{2 \psi_i + 1} \leq \frac{2 \psi_{i-1} - 1}{2\psi_i + 1} \frac{1}{\phi_n (2 \psi_{i-1} + 1)} \leq \frac{1}{\phi_n (2 \psi_{i-1} + 1)}
\end{align*}
where we use our induction hypothesis and then the fact that $\psi_i$ is monotonically increasing.

For the second inequality \eqref{Eqn:BetaLemma2}, applying our identity \eqref{Eqn:BetaIdentity} yields
\begin{equation*}
    (2 \psi_j + 1)\beta_j = (2 \psi_{j-1} - 1) \beta_{j-1} \leq (2 \psi_{j-1} + 1) \beta_{j-1}
\end{equation*}
and the result follows from induction.
\end{proof}

Now we begin verifying our inequalities. Since we are restricted to gradient span algorithms, we know that $x_i \in \text{span}\{g_0, \dots, g_{i-1}\}$ for all $i$. We will frequently use the fact from our orthonormal construction of $g_i$ that for any $j \geq n$, if $i<j$ we have $\langle g_j, g_i \rangle = 0$ and $\langle g_j, x_i - x_j \rangle = 0$. We begin with our inequalities $Q_{i,j}(L)$ for $i<j$.

\begin{itemize}
    \item Let $i < j < n$. Then by the definition of our oracle family $\mathcal{O}_L^\mathcal{H}$, we know $Q_{i,j}(L) \geq 0$.
    \item Let $i < n \leq j$. Then we have
\begin{align*}
    Q_{i,j}(L) &= f_i - f_j - \langle g_j, x_i - x_j \rangle - \frac{1}{2L}\|g_i - g_j\|^2 \\
    &= f_i - \frac{1}{2L}\|g_i\|^2 - f_j - \frac{1}{2L}\|g_j\|^2 \\
    & \geq v_m - f_j - \frac{1}{2L} \|g_j\|^2 \\
    & = \frac{L \sigma}{2} \left( \frac{1}{\phi_n} - (2 \psi_j + 1) \beta_j \right) \\
    & \geq 0
\end{align*}
where the first inequality follows from the definition of $m$ and the second inequality follows from \cref{Lem:BetaInequalities}. Note that in the special case $i=m$ and $j=n$, we have equality: $Q_{m,n}(L) = 0$.

\item Let $n \leq i < j$. Then we have
\begin{align*}
    Q_{i,j}(L) &= f_i - f_j - \langle g_j, x_i - x_j \rangle - \frac{1}{2L}\|g_i - g_j\|^2 \\
    &= f_i - \frac{1}{2L}\|g_i\|^2 - f_j - \frac{1}{2L}\|g_j\|^2 \\
    & = \frac{L \sigma}{2} \left( (2 \psi_i - 1) \beta_i - (2 \psi_j + 1) \beta_j  \right) \\
    & \geq 0
\end{align*}
where the inequality follows from \cref{Lem:BetaInequalities}. Note that in the special case $i=j-1$, from \eqref{Eqn:BetaIdentity} we have equality: $Q_{j-1,j}(L) = 0$.
\end{itemize}

Now consider the inequalities $W_{\star,i}$. 
\begin{itemize}
    \item Let $i<n$. Then
\begin{align*}
    W_{\star,i} &= f_\star - f_i - \langle g_i, x_\star - x_i \rangle \\
    & = f_\star - f_i - \langle g_i, z' - x_i \rangle \\
    & \geq 0.
\end{align*}
In the second line, we use the fact that $x_\star = z_{N+1} = z' - \sum_{\ell = n}^N \frac{\psi_\ell}{L} g_\ell$, along with the orthogonality of our $g_\ell$ vectors, and the inequality follows from our dual feasibility expression \eqref{Eqn:DualFeas2}.

\item Let $i \geq n$. Then we have
\begin{align*}
    W_{\star,i} &= f_\star - f_i - \langle g_i, x_\star - x_i \rangle \\
    & = f_\star - f_i - \langle g_i, z_{N+1} - x_i \rangle \\
    & = f_\star - f_i + \frac{\psi_i}{L}\|g_i\|^2 \\
    & = -L \sigma \psi_i \beta_i + \frac{\psi_i}{L} L^2 \sigma \beta_i \\
    &= 0.
\end{align*}
\end{itemize}

Next, we verify that our inductive hypotheses $U_i \geq 0$ hold at each step in our algorithm history. For $i<n$, we have
\begin{align*}
    U_i & = \tau_i\left(f_\star - f_i + \frac{1}{2L}\|g_i\|^2 \right) + \frac{L}{2}\|x_0 - x_\star\|^2 - \frac{L}{2}\|z_{i+1} - x_\star\|^2 \\
    & = \tau_i\left(f_\star - f_i + \frac{1}{2L}\|g_i\|^2 \right) + \frac{L}{2}\|x_0\|^2 - \frac{L}{2}\|z_{i+1}\|^2 + L \langle z_{i+1} - x_0 , x_\star \rangle \\
    & = \tau_i\left(f_\star - f_i + \frac{1}{2L}\|g_i\|^2 \right) + \frac{L}{2}\|x_0\|^2 - \frac{L}{2}\|z_{i+1}\|^2 + L \langle z_{i+1} - x_0 , z' \rangle \\
    & \geq 0.
\end{align*}
The third equality again follows from $x_\star = z_{N+1} = z' - \sum_{\ell = n}^{N} \frac{\psi_\ell}{L} g_\ell$ and the orthogonality of our vectors $g_\ell$, and the inequality follows from our dual feasibility expression \eqref{Eqn:DualFeas1}.

Thus we have shown that $O \in \mathcal{O}_L^\mathcal{H}$. All that remains to be shown is our lower bound. We claim that it is sufficient to show that $U_N = 0$. By the definition of $U_N$ we can then rearrange to obtain
\begin{align*}
    f_N - f_\star & = \frac{1}{\tau_{n,N}} \left(\frac{L}{2}\|x_0 - x_\star\|^2 - \frac{L}{2}\|z_{N+1} - x_\star\|^2 \right) \\
    & = \frac{1}{\tau_{n,N}} \left(\frac{L}{2}\|x_0 - x_\star\|^2 - \frac{L}{2}\|x_\star - x_\star\|^2 \right) \\
    & = \frac{L\|x_0 - x_\star\|^2}{2 \tau_{n,N}}.
\end{align*}
Therefore, the lower bound 
\begin{equation}
    \min_{A\in \mathcal{A}^\mathcal{H}} \max_{O \in\mathcal{O}_L^\mathcal{H}} \frac{f(x_N)-f(x_\star)}{\frac{L}{2}\|x_0-x_\star\|^2} \geq \frac{1}{\tau_{n,N}}
\end{equation}
holds.

To show that $U_N = 0$, we first look at $U'$:
\begin{align*}
    U' &= \phi_n\left(f_\star - f_m + \frac{1}{2L}\|g_m\|^2\right) + \frac{L}{2}\|x_0-x_\star\|^2 - \frac{L}{2}\|z' - x_\star\|^2 \\
    & = -\frac{L}{2}\|z'-x_0\|^2 + \frac{L}{2}\|x_0-x_\star\|^2 - \frac{L}{2}\|z' - x_\star\|^2 \\
    &= L \langle x_0-z', z' - x_\star \rangle \\
    &= L \langle x_0-z', \sum_{\ell = n}^N \frac{\psi_\ell}{L} g_\ell \rangle \\
    & = 0
\end{align*}
where the final equality comes from the fact that $x_0 - z' \in \text{span}(g_0, \dots, g_{n-1})$ and the orthogonality of $g_\ell$.
Finally, as noted above, we have $Q_{m,n}(L) = 0$, and $Q_{i-1,i}(L) = W_{\star,i} = 0$ for all $i \geq n$. Then following the standard OBL induction, we have
\begin{align*}
    U_N &= U_n + \sum_{\ell=n+1}^N \left(\phi_\ell Q_{\ell-1,\ell}(L) + \psi_\ell W_{\star,\ell} \right) \\
     & = U' + \phi_n Q_{m,n}(L) + \psi_n W_{\star,n} + \sum_{\ell=n+1}^N \left(\phi_\ell Q_{\ell-1,\ell}(L) + \psi_\ell W_{\star,\ell} \right) \\
     &= 0
\end{align*}
and our proof is complete.

\subsection{Additional Numerical Results} \label{SubSec:AdditionalNumerics}

For completeness, we include performance plots for each individual problem type considered in \cref{SubSec:Synth}.

\begin{figure}[ht]
    \centering
    \includegraphics[width=1.0\textwidth]{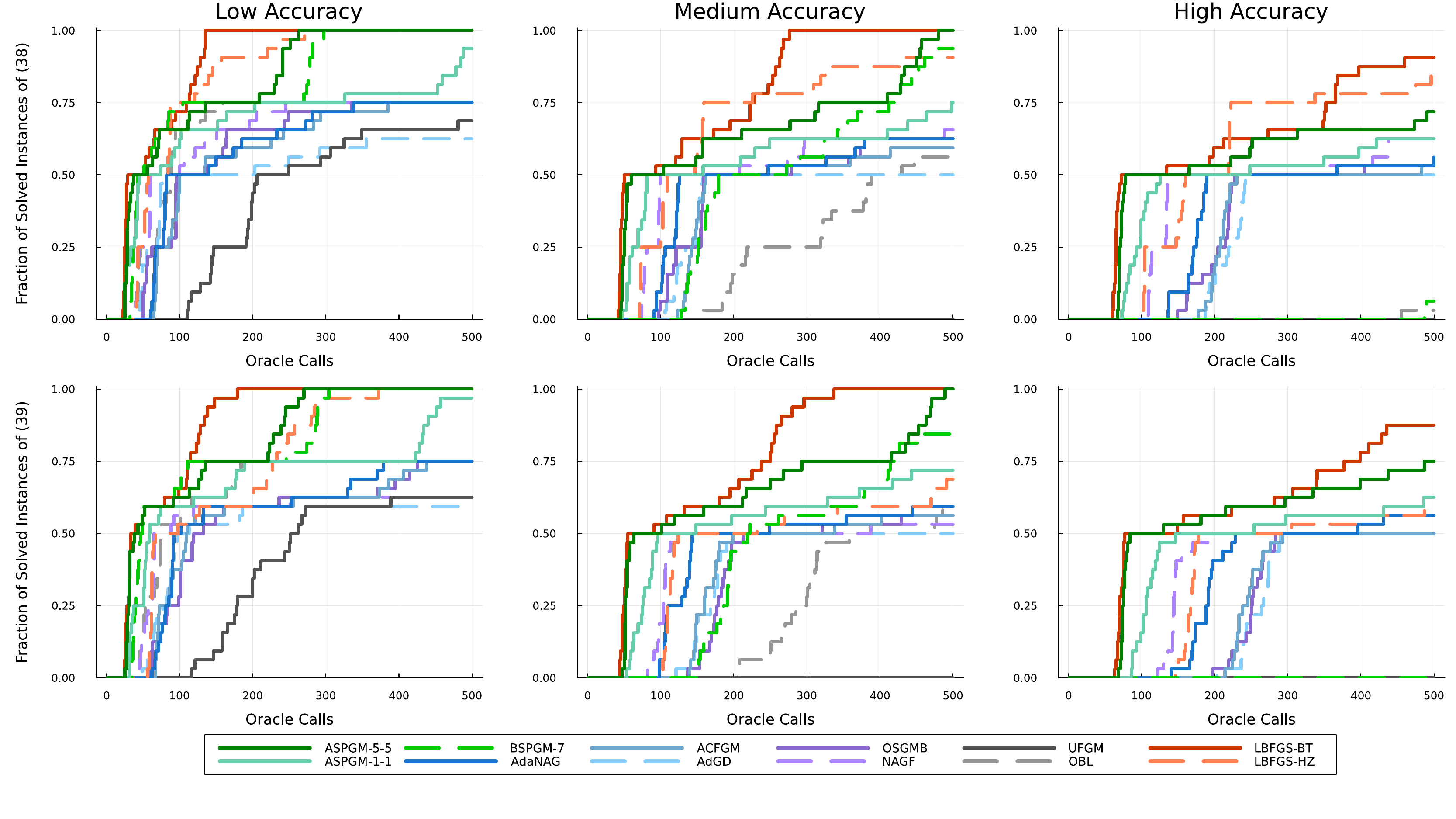}
    \caption{Performance comparison for least squares regression \eqref{Eqn:Objective_LSReg} and logistic regression \eqref{Eqn:Objective_LogReg}. The performance is measured by $(f(x_n) - f_\star)/(f(x_0) - f_\star)$ and the target accuracies from left to right are $10^{-4}, 10^{-7}, 10^{-10}$.}
\end{figure}

\begin{figure}
    \centering
    \includegraphics[width=1.0\textwidth]{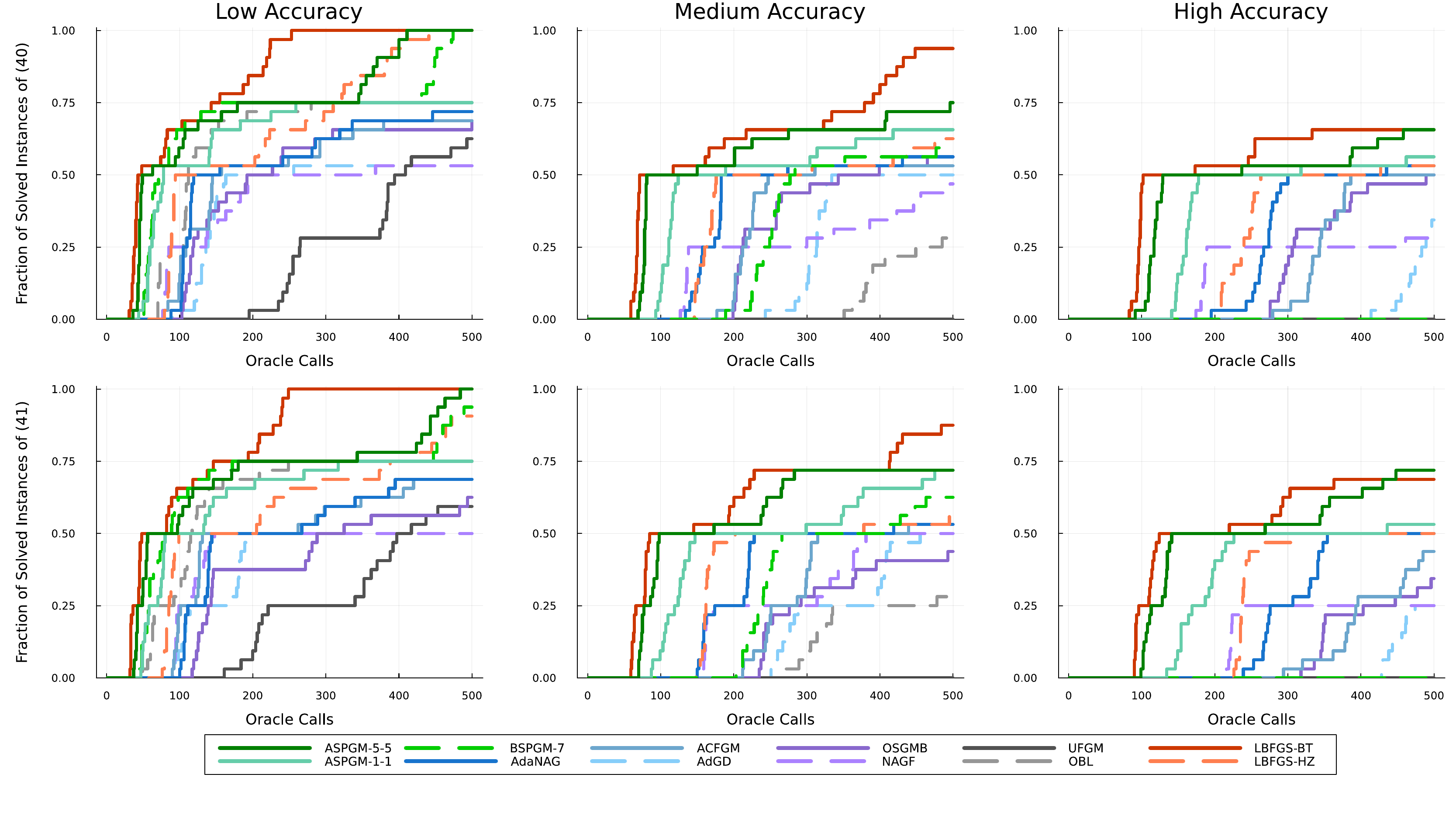}
    \caption{Performance comparison for smooth feasibility problems \eqref{Eqn:Objective_LogSumExp} and \eqref{Eqn:Objective_PosSquared}. The performance is measured by $(f(x_n) - f_\star)/(f(x_0) - f_\star)$ and the target accuracies from left to right are $10^{-4}, 10^{-7}, 10^{-10}$.}
\end{figure}

\begin{figure}
    \centering
    \includegraphics[width=1.0\textwidth]{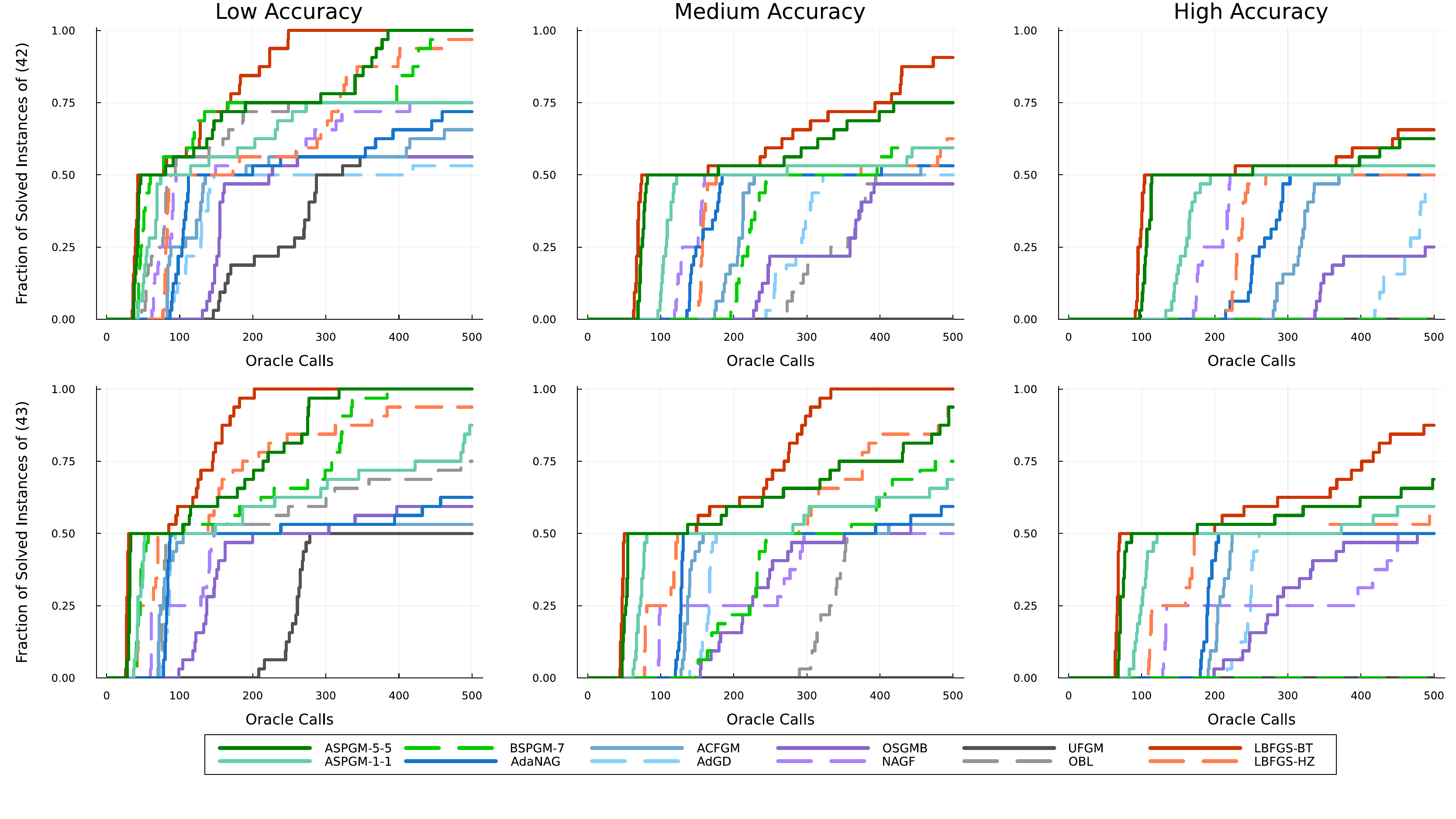}
    \caption{Performance comparison for locally smooth problems \eqref{Eqn:Objective_4Norm} and \eqref{Eqn:Objective_CubicReg}. The performance is measured by $(f(x_n) - f_\star)/(f(x_0) - f_\star)$ and the target accuracies from left to right are $10^{-4}, 10^{-7}, 10^{-10}$.}
\end{figure}

\end{document}